\documentclass[reqno]{amsart}
\usepackage{graphicx}
\usepackage{epsfig}
\usepackage{amssymb}
\usepackage{xcolor}
\usepackage[
backend=biber,
style=alphabetic,
sorting=nyt
]{biblatex}
\usepackage{amsfonts}
\usepackage{dsfont}
\usepackage{amsmath}
\usepackage{stmaryrd}
\usepackage{hyperref}
\usepackage[capitalise,nameinlink]{cleveref}
\usepackage{subfigure, verbatim}
\newtheorem{conjecture}{Conjecture}
\newtheorem{theorem}{Theorem}
\newtheorem*{theorem*}{Theorem}
\newtheorem{lemma}{Lemma}
\newtheorem{proposition}{Proposition}[section]
\theoremstyle{definition}

\newtheorem{corollary}{Corollary}

\theoremstyle{remark}
  
\numberwithin{equation}{section}
\newcommand\numberthis{\addtocounter{equation}{1}\tag{\theequation}}
\theoremstyle{remark}

\newcommand{\beq}{\begin{equation}}
\newcommand{\eeq}{\end{equation}}
\newcommand{\beqn}{\begin{equation*}}
\newcommand{\eeqn}{\end{equation*}}
\newcommand{\bea}{\begin{eqnarray}}
\newcommand{\eea}{\end{eqnarray}}
\newcommand{\bean}{\begin{eqnarray*}}
\newcommand{\eean}{\end{eqnarray*}}

\newcommand{\be}{\begin{enumerate}}
\newcommand{\ee}{\end{enumerate}}
\newcommand{\bi}{\begin{itemize}}
\newcommand{\ei}{\end{itemize}}
\newcommand{\bd}{\begin{description}}
\newcommand{\ed}{\end{description}}

\begin{document}
\title[The $\beta=2$ partition function in the $q$-aspect]{The $\beta=2$ partition function of Dirichlet $L$-functions and typical maxima}
\author{Christopher Atherfold}
\address{School of Mathematics, University of Bristol, Bristol, BS8 1UG, United Kingdom}
\email{christopher.atherfold@bristol.ac.uk}
\begin{abstract}
    We study the $\beta=2$ partition function $\int_{|h| \leq \log ^{\theta} q/2}|L(1/2+ih,\chi)|^2dh$ for typical Dirichlet characters modulo a large prime $q$ and $\theta \in (-1/3,0]$ motivated by a $q$-analogue of the Saksman--Webb conjectures. When $\theta <0$, we use an extension of Harper's randomisation argument with different localisations to recover moment upper bounds consistent with the critical normalisation predicted there. As an application, we prove that for $q(1-o(1))$ Dirichlet characters modulo $q$ with $U(q) \to \infty$, $$\max_{|h| \leq 1/2}|L(1/2+ih,\chi)|\ll \frac{e^{U(q)}\log  q\log_3 q(\log_4 q)^{1/2}}{(\log_2 q)^{3/4}},$$ establishing an upper bound matching the predictions of the $q$-analogue of the Fyodorov--Hiary--Keating conjectures up to second order. Our proofs adapt Harper's argument in \cite{harper2019partition}.
\end{abstract}
\maketitle
\section{Introduction} \subsection{Motivation and fixed intervals}
We study how large $L$-functions can be for typical Dirichlet characters modulo a large prime $q$ near the central point through the lens of the $\beta=2$ partition function. This is given by
\begin{equation}
    \mathcal{Z}_2(q,\chi,\theta)=\int_{|h| \leq \log ^{\theta} q/2}|L(1/2+ih,\chi)|^2dh
\end{equation}
where $\theta \in (-1/3,0]$. Such a family of characters is a discrete unitary family, whereas if we replace $\chi$ with $n^{it}$ and $q$ with $T$, we obtain the Riemann zeta function, which is a continuous unitary family. These quantities are related to the typical maxima on constant length intervals, which we will discuss soon.\par
We begin with the $\beta=2$ partition function. In \cite{sw20}, Saksman and Webb investigate many distributional questions concerning random Euler products which they conjecture should also hold for the Riemann zeta function on typical intervals. Among their conjectures is 
\begin{equation} \label{sw_zeta}
    \frac{\sqrt{\log \log T}}{\log T}\int_0^1|\zeta(1/2+i(\tau+h))|^2dh \underset{T \rightarrow \infty}{\longrightarrow} \int_0^1g(h)\lambda_2(dh)
\end{equation}
in distribution, where $\tau$ is a uniformly distributed random variable on the interval $[T,2T]$ and $\lambda_2(dh)$ is the critical Gaussian multiplicative chaos measure. The function $g$ is a positive regular random function which is well behaved in various senses, and for more details one can see the theorem statements there. This conjecture will be settled in forthcoming work of Harper, Saksman and Webb. It is natural to ask if the analogue of such a conjecture holds for unitary averages over Dirichlet characters. We present the $q$-analogue of Conjecture $1.11$ in \cite{sw20}.
\begin{conjecture} \label{q-aspect_partition_conjecture}
Let $q$ be a large prime and $\chi_q$ be a non-trivial Dirichlet character modulo $q$ chosen uniformly at random. Then
    $$\frac{\sqrt{\log \log  q}}{\log  q}\int_0^1|L(1/2+ih,\chi_q)|^2dh \underset{q \rightarrow \infty}{\longrightarrow} \int_0^1g(h)\lambda_2(dh),$$
    where $g(h)$ and $\lambda_2(dh)$ are the same as identified in equation (\ref{sw_zeta}). 
\end{conjecture}
As mentioned in \cite{harper2019partition}, this can be described as the $\beta=2$ partition function of $\log |L(1/2+ih,\chi)|$ for fixed $\chi$ and varying $h$. Two striking features of the conjecture are the normalisation factor and the non-Gaussian limit. Before determining the limiting distribution, we need to have the correct normalisation. Our first result provides evidence towards this being the correct normalisation factor, together with the work of Harper in \cite{harper2023typicalsizecharacterzeta} and the recent work of Vihko \cite{vihko2025dirichlet} where Vihko proves that a uniformly random Dirichlet $L$-function on the critical line converges as a random Schwartz distribution to the Steinhaus random Euler product as $q \to \infty$ where $q$ is restricted to prime values. The results in \cite{harper2023typicalsizecharacterzeta} are particularly revealing, since Harper connects his results on random multiplicative functions to low moments of character sums through an $L^2$ randomisation estimate. In the setting of Conjecture \ref{q-aspect_partition_conjecture}, we obtain the following theorem.
\begin{theorem}
Suppose $q$ is a large prime. Then uniformly for $r \in [0,1]$,
    $$\frac{1}{q-1}\sum_{\chi \bmod{q}}(\mathcal{Z}_2(q,\chi,0))^r \ll \big(\frac{\log  q}{1+(1-r)\sqrt{\log \log  q}}\big)^r.$$
    \label{pf_moments_constant}
\end{theorem}
By Markov's inequality, we obtain the following corollary, with the proof deferred to Section \ref{pfm_proof}.
\begin{corollary}\label{L_constant}
    \textit{For $\theta=0$, we have uniformly for $q$ a large prime and $\lambda \ge 2$}
    \begin{align*}
        \frac{1}{q-1}\#\big\{\chi \bmod{q}:\mathcal{Z}_2(q,\chi,0) > \frac{\lambda \log  q}{\sqrt{\log \log  q}}\big\} 
        \ll \frac{\min\{\log  \lambda,\sqrt{\log \log  q}\}}{\lambda}.
    \end{align*}
\end{corollary}
\subsection{Partition functions on mesoscopic intervals}
We study these on intervals of size $\log ^{\theta} q$ for $\theta \in (-1/3,0]$. This introduces a new difficulty, where we need to introduce explicit conditioning on the contribution of $P_\theta = \exp(\log ^{|\theta|}q)$ smooth integers. We have for $P_\theta$ smooth integers, $n^{ih}$ is essentially frozen up to $O(1)$ fluctuations. Consequently, we need to condition on the value of the Dirichlet polynomial over the $P_\theta$-smooth numbers to obtain the correct order of magnitude. This has previously been studied in Arguin--Hamdan \cite{ah24} and before that in \cite{aor21}. Our approach to this is rather different. Previously, the authors used level sets of $\log |\zeta|$ to derive their results, whereas we are integrating $L$ directly, which necessitates the explicit conditioning.\par
In \cite{harper2023typicalsizecharacterzeta}, Harper can implement the \textit{coarser conditioning} implicitly at one localisation scale. For our mesoscopic question, we need multiple localisation scales to apply both the conditioning on the slowly varying primes and the linear barrier on the log-correlated part of the Dirichlet polynomial. We do this by using an (explicit) one point conditioning statement with Harper's aforementioned randomisation principle, where we can allow the conditioning polynomial to incorporate different localisation scales and approximating errors simultaneously. Take
\begin{equation} \label{short_lognormal}
        Y_{P_\theta,\chi}(h) := \Re \big(\sum_{p \leq P_\theta}\big(\frac{\chi(p)}{p^{1/2+ih}}+\frac{(\chi(p))^2}{2p^{1+2ih}}\big)\big).
\end{equation}
Note that $Y_{P_\theta,\chi}(h)$ is asymptotically Gaussian with mean $0$ and variance $|\theta|\log\log q/2$ when $\theta < 0$. When we look at an $L$-function, we have to exponentiate $Y_{P_\theta,\chi}(0)$, and consequently conditioning on the frozen part is very effective given this is a natural $L^2$ quantity. For the randomisation result, we also need to restrict $W \ge W_0$ where $W_0$ is some fixed negative real number. We stress that these are restricted non-normalised expectations, and this is different to the setting in \cite{ah24}.
\begin{theorem}\label{pf_moments_meso}
    Fix $\theta \in (-1/3,0]$ and let $P_\theta = \exp(\log ^{|\theta|}q)$. Then, uniformly for $q$ a large prime and $r \in [0,1]$, if $\theta <0$ and $W \ge W_0$
    \begin{equation*}
        \frac{1}{q-1}\sum_{\chi \bmod{q}}\mathbf{1}_{Y_{P_\theta,\chi}(0) \in [W,W+1]}(\mathcal{Z}_2(q,\chi,\theta))^r\ \ll \big(\frac{e^{2W}\log ^{1+2\theta} q}{1+(1-r)\sqrt{\log \log  q}}\big)^r.
    \end{equation*}
    We also have for $W \ge W_0$
    \begin{equation*}
        \frac{1}{q-1}\sum_{\chi \bmod{q}}\mathbf{1}_{W_0 < Y_{P_\theta,\chi}(0) \leq W}(\mathcal{Z}_2(q,\chi,\theta))^r \ll \big(\frac{e^{2W}\log ^{1+2\theta} q}{1+(1-r)\sqrt{\log \log  q}}\big)^r.
    \end{equation*}
\end{theorem}
The second estimate follows from the first by considering the $r \in [2/3,1]$ range and splitting the interval $(W_0,W]$ into subintervals of unit length. Then one applies the first statement in Theorem \ref{pf_moments_meso} to each subinterval (and uses H\"older's inequality to obtain the moments for $r\in [0,2/3)$). The conditioning in the second statement is much closer to a barrier event on the frozen field, which are commonly used in the context of log-correlated fields and branching random walks. As a corollary of Theorems \ref{pf_moments_constant} and \ref{pf_moments_meso}, we obtain the following tail bounds.
\begin{corollary} \label{typical_partition_meso}
    \textit{Fix $\theta \in (-1/3,0)$. Then, uniformly for a large prime $q$ and $2 \le \lambda \le e^{\sqrt{\log\log q}}$ with $W \ge W_0$}
     \begin{align*}
        \frac{1}{q-1}\#\big\{\chi \bmod{q}:\mathbf{1}_{ Y_{P_\theta,\chi}(0) \leq W}\mathcal{Z}_2(q,\chi,\theta) > \frac{\lambda e^{2W} \log ^{1+2\theta} q}{\sqrt{\log \log  q}}\big\} 
        \ll \frac{\log  \lambda}{\lambda}.
    \end{align*}
\end{corollary}
Since the key aspect of Harper's randomisation argument is to use the orthogonality of the characters, we also get the analogous statement for zeta where the adaptations required will be discussed in Section \ref{pfm_proof}.
\begin{corollary}\label{typical_partition_zeta}
    \textit{Fix $\theta \in (-1/3,0)$, let $T$ be large and with $P_\theta =\exp(\log^{|\theta|}T)$}, define
    $$Y_{P_\theta,t}(h)=\Re \big(\sum_{p \leq P_\theta}\big(\frac{p^{-it}}{p^{1/2+ih}}+\frac{p^{-2it}}{2p^{1+2ih}}\big)\big), \quad \mathcal{P}_2(t,\theta) = \int_{|h| \le \log ^{\theta}T/2}|\zeta(1/2+i(t+h))|^2dh.$$
    \textit{Then uniformly for $2 \le \lambda \le e^{\sqrt{\log\log T}}$ with $W \ge W_0$}
    $$\frac{1}{T}\text{meas}\Big\{t \in [T,2T]:\mathbf{1}_{ Y_{P_\theta,t}(0) \le W}\mathcal{P}_2(t,\theta) > \frac{\lambda e^{2W}\log ^{1+2\theta} T}{\sqrt{\log \log  T}}\Big\} \ll \frac{\log  \lambda}{\lambda}.$$
\end{corollary}
We defer the proof of these corollaries to Section \ref{pfm_proof}. Corollary \ref{L_constant} indicates that the normalisation in Conjecture \ref{q-aspect_partition_conjecture} is the correct order of magnitude. As for Corollary \ref{typical_partition_zeta}, this is in some sense an improvement on Theorem 1.3 in \cite{ah24}, though this must be qualified. Arguin and Hamdan study the $\beta=2$ partition function of $\zeta$ when mollified randomly, namely
$$\frac{1}{T}\text{meas}\Big\{t \in [T,2T]:\prod_{p \le P_\theta}|1-p^{-1/2-it}|^2\mathcal{P}_2(t,\theta)>\frac{\lambda \log^{1+2\theta} T}{\sqrt{\log\log T}}\Big\}.$$
Therefore they have no need to consider $W$, and can obtain their results for $\theta \in (-1,0]$ rather than our $\theta \in (-1/3,0)$ range. So we sacrifice generality in $\theta$ and with the need to introduce conditioning $W$, but we do get an improvement of $\ll \frac{\log \lambda}{\lambda}$ bound on the right tail compared to the $\ll \lambda^{-1+\varepsilon}$ right tail bound in \cite{ah24}.
\par
The appearance of the $\sqrt{\log \log  q}$ in the theorems is connected to the exponent $\beta=2$, the critical exponent in the language of multiplicative chaos. At the critical exponent, there are no longer enough sufficiently high points to support a non-degenerate measure, so we require the Seneta--Heyde normalisation to obtain something meaningful, which corresponds to the $\sqrt{\log \log  q}$ correction, which simply is not visible for smaller $\beta$.
\subsection{Typical constant length interval maxima}
The Fyodorov--Hiary--Keating conjectures were first proposed in \cite{fhk12} and explored further in \cite{fk14}. These have been studied rigorously in many works since these papers, with the state of the art due to Arguin--Bourgade--Radziwi\l{}\l{} in \cite{abr20} and \cite{abr23}. Among other results, they prove for large $T$, one has
$$\frac{1}{T}\text{meas}\Big\{t \in [T,2T]:\max_{|h| \leq 1}|\zeta(1/2+i(t+h))| > \frac{e^U\log T}{(\log \log T)^{3/4}}\Big) \asymp Ue^{-2U}$$
for $U \in [10,\sqrt{\log \log T}]$. This confirms order of magnitude for the size of the right tail decay. As for the latest developments in the random matrix setting, see the work of Paquette and Zeitouni \cite{pz22}. We propose the following conjecture.
\begin{conjecture} \label{constant_qFHK}
    Suppose $q$ is a sufficiently large prime. Then one has that for $1 \leq U \le \sqrt{\log \log  q}$,
    \begin{equation} \label{q_count}
        \frac{1}{q-1}\#\big\{\chi \bmod{q}: \max_{|h| \leq 1/2}|L(1/2+ih,\chi)| \geq \frac{e^U\log  q}{(\log \log  q)^{3/4}} \big\} \asymp Ue^{-2U}.
    \end{equation}
\end{conjecture}
While the work of Harper in \cite{harper2023typicalsizecharacterzeta} predicts the normalisation in Conjecture \ref{q-aspect_partition_conjecture}, it is unclear from previous literature that such a conjecture should hold as the typical maxima are significantly more delicate. In our final theorem, we prove that this conjecture is sharp in the upper bound up to the second order correction term, slightly improving the work of Harper.
\begin{theorem}
    Suppose $q$ is a large prime, $1 \leq U \leq \log \log  q$. Then 
    \begin{align*}
        \#\big\{\chi \bmod{q}: \max_{|h| \leq 1/2}|L(1/2+ih,\chi)| \geq \frac{e^U\log  q}{(\log \log  q)^{3/4}} \big\} \ll qUe^{-2U}(\log \log \log  q)^2.
    \end{align*}
    \label{short_range_max}
\end{theorem}
In particular, for any $U(q) \to \infty$, we have
$$\#\big\{\chi \bmod{q}: \max_{|h| \leq 1/2}|L(1/2+ih,\chi)| \geq \frac{e^{U(q)}\log  q\log _3q(\log_4q)^{1/2}}{(\log_2  q)^{3/4}}\big\} =o(q),$$
where $\log_k$ denotes the $k$-fold iterated logarithm.
\subsection{Proof ideas}
There are two main components to each proof: a randomisation argument and a probabilistic analysis of the resulting random Euler products. In all our theorems, we extend Harper's approach in \cite{harper2019partition} to this different setting and explain the differences used in our approach.\par
In all of our proofs, we study expressions of the form
\begin{equation}\label{heuristic_quantity}
    \sum_{\chi \bmod{q}}\mathbf{1}_{\mathcal{G}_\chi}\int_{|h| \leq 1/2}|\sum_{\substack{m \leq q^{\varepsilon} \\ P^+(m) \leq P}}\frac{\chi(m)}{m^{1/2+ih}}\sum_{\substack{n \leq q^{1/2-2\varepsilon} \\ P^-(n) > P}}\frac{\chi(n)}{n^{\sigma+ih}}|^2dh
\end{equation}
where $P$ is a suitable parameter ($P$ is chosen such that $\log \log P \asymp \log \log  q$) and $\mathcal{G}_\chi$ is an event which occurs with high probability which controls the size of the Dirichlet polynomial at many scales. Indeed, the key observation is after passing to sums involving Steinhaus random multiplicative functions, the size of the conditioned integral of random Dirichlet polynomial over the $P$ smooth integers by $\mathbf{1}_{\mathcal{G}_f}$ is $\ll \log P/\sqrt{\log \log P}$. This is a consequence of the ballot theorem which is sufficient for the bounds in Theorems \ref{pf_moments_constant} and \ref{pf_moments_meso}.\par
The further saving of $\log \log P$ seen in Theorem \ref{short_range_max} comes from strengthening the event $\mathcal{G}_\chi$ to also include a condition that ensures that $$|\sum_{\substack{m \leq q^{\varepsilon} \\ P^+(m) \leq P}}\frac{\chi(m)}{m^{1/2+ih}}| > \frac{\log P}{V}$$
for some parameter $V$. Since we are looking at very large values, one conditions on the end point in order to get a greater saving, so we conclude from there. The reason why we transfer to the random side is because we then gain independence of the distinct prime ranges once again, even with the barriers applied (which is lost when taking restriction over characters).\par
As for Theorem \ref{pf_moments_meso}, we need some further considerations. The contributions from integers which are $P_\theta$-smooth vary very slowly over the mesoscopic window, so need to be conditioned upon a meaningful upper bound. This necessitates a new character randomisation where we can localise at different scales with different approximation errors so we can condition on the primes less than $P_\theta$ and place the barrier events for the log-correlated regime simultaneously. This is different to \cite{harper2023typicalsizecharacterzeta} for example where barriers can be placed after the initial transfer, whereas since our barriers are on the integral itself, they must be randomised immediately. Since the conditioning on the frozen environment is directly applied to $L^2$ quantities, then it is very effective and allows for us to condition only on one point per mesoscopic window. The methods used here allow for the low moments of the partition function to retain some uniformity in $r$, which leads to the sharper right tail estimates in Corollaries \ref{typical_partition_meso} and \ref{typical_partition_zeta} as the barriers are adapted to retain this transition.\par
The large values argument differs from \cite{harper2019partition} slightly, in that we allow the Gaussian comparison to start later on when analysing $|L(1/2+ih,\chi)|$. We instead begin the comparison at a starting scale depending on $U$ rather than being fixed. This in addition to introducing the discretisation on the lower bound for the smooth contribution to $L$ only when we have connected it to an expression involving random Euler products and a slightly different ballot theorem application lead to the improvement over \cite{harper2019partition}.
\subsection{Further discussion of results} \label{further_disc}
Given our discussion of Theorems \ref{pf_moments_constant} and \ref{pf_moments_meso}, we remark on Conjecture \ref{constant_qFHK} and Theorem \ref{short_range_max}. Following the work of Arguin--Bourgade--Radziwi\l{}\l{}, it should be possible to recover Conjecture \ref{constant_qFHK} in its entirety if one adapts the work in \cite{abr20} and \cite{abr23} to the $q$-aspect. Since we focus on an upper bound in this paper, the arithmetic input in \cite{abr20} could be mimicked by computing a similar twisted fourth moment in the $q$-aspect. However, to directly implement their method, one would need a lot stronger arithmetic information than we use here in the character analogue of their twisted fourth moment estimate, requiring strong bounds uniformly in four separate shifts including slightly off the critical line. We still achieve the second order correction term in Theorem \ref{short_range_max} only with a simple conditional fourth moment estimate. \par
As for improvements to the results, the most significant one is to Theorem \ref{short_range_max} since we resort to approximating a Dirichlet polynomial over rough numbers by Cauchy--Schwarz, which results in forcing one of the parameters $V$ to be taken as too large later on in the proof, whereas one might hope this averages out. The other issue is a slight loss on the sharpness of the right tail upper bound due to Markov's inequality in Corollary \ref{L_constant}. We conjecture that for $\theta =0$ and fixed $\lambda \geq 2$, then one has
$$\# \big\{\chi \bmod{q}: \mathcal{Z}_2(q,\chi,0) > \frac{\lambda \log  q}{\sqrt{\log \log  q}}\big\} \asymp \frac{q}{\lambda}$$
in line with similar types of results obtained in \cite{barral15}, which would likely require Conjecture \ref{q-aspect_partition_conjecture} to be proved fully. Removing the restriction of $W_0 \le W$ would also be very desirable in Theorem \ref{pf_moments_meso}, but this appears rather difficult given the style of our anatomy mollifier.\par
Finally, we discuss matching lower bounds for our theorems. For Theorem \ref{short_range_max}, following the work of \cite{abr23} could lead to a lower bound at the precision predicted by Conjecture \ref{constant_qFHK} as their approach is very probabilistic in nature. However, proving an order of magnitude result comparable to Theorem \ref{pf_moments_constant} is a different challenge. This would either require a proof of Conjecture \ref{q-aspect_partition_conjecture}, or deploying a strategy similar to that in \cite{harper26}. These need rather different techniques to the ones developed here, so we choose not to pursue them.
\subsection{Notation and organisation of the paper}
Due to the different notions of averaging throughout the paper, we summarise them here. We use Steinhaus random multiplicative functions very frequently in this paper. Define the sequence $(f(p))$ to be independent identically distributed uniform random variables on the complex unit circle and extended completely multiplicatively to the natural numbers yielding $f(n)$. We denote the Steinhaus random Euler product by
$$F_y(s) = \prod_{p \le y}\Big(1-\frac{f(p)}{p^s}\Big)^{-1}.$$
We will use $\mathbb{E}^{\text{char}}$ to denote averaging over all Dirichlet characters modulo $q$, and $\mathbb{E}^f$ to refer to averages involving random multiplicative functions, and when we want to condition on the value of the short Dirichlet polynomial $Y_{P_\theta,\chi}(0)$ (or $Y_{P_\theta,f}(0)$) to be in the interval $[W,W+1]$ we denote this by the non-renormalised expectation $$\mathbb{E}_{W}^{\text{char}}[X] = \mathbb{E}^{\text{char}}[X\mathbf{1}_{Y_{P_\theta,\chi}(0) \in [W,W+1]}] $$ (or $\mathbb{E}_W^f$) respectively. We also define the expectation
$$\mathbb{E}^{\text{cont}}[X] = \frac{1}{T\hat{\Phi}(0)}\int_\mathbb{R}X(t)\Phi(t/T-1)dt$$
where $\Phi$ is a Beurling--Selberg function which pointwise dominates $\mathbf{1}_{[0,1]}$ and has its Fourier transform supported on $[-1/2\pi,1/2\pi]$. Such a $\Phi$ is constructed in \cite{V85}. We do this for the various times that we want to describe how the $\zeta$ version of Theorem \ref{pf_moments_meso} is proved, and this allows us to regain perfect orthogonality. We have the inequality 
$$\frac{1}{T}\int_{T}^{2T}|\sum_{n \le x}n^{-1/2-it}|^{2q}dt \le \frac{1}{T} \int_{\mathbb{R}}\Phi(t/T-1)|\sum_{n \le x}n^{-1/2-it}|^{2q}dt \ll \mathbb{E}^{\text{cont}}|\sum_{n \le x}n^{-1/2-it}|^{2q},$$
with the same holding for statements involving $\mathbb{E}_W^{\text{cont}}[D(s)]$ where $D(s)$ is a Dirichlet polynomial. The advantage here is that for $1 \le m,n \le T$, $\mathbb{E}^{\text{cont}}[m^{it}n^{-it}] = \mathbf{1}_{m=n}$, which preserves the perfect orthogonality we have with the characters, so there are no error terms arising from this situation, and allows for character averaging and $t$-averaging to be interchangeable with the methods we present.
There are also times when we want to understand the size of a set where an event occurs for characters, so we let $\mathbb{P}_q$ denote the uniform probability measure over the set of Dirichlet characters mod $q$. Since we are also considering the mesoscopic regime, we fix some quantities which occur frequently. For a given value of $\theta$, we let $P_\theta = \exp(\log ^{|\theta|}q)$ and $\varepsilon_\theta = \frac{\log ^\theta(q)}{(\log \log  q)^2}$ (which we observe is $\gg (\log  q)^{-1/2}$ for the range of $\theta$ we are investigating). Finally we discuss Dirichlet polynomials where the contributions come from integers with certain factorisation properties throughout the paper. We use $P^+(n)$ and $P^{-}(n)$ to denote the largest and smallest prime factor of $n$ respectively.\par
We collect most of the preliminary results we need to prove our theorems in Section \ref{sec2}, and keep the probability estimates for Appendix \ref{prob_heur}. We then prove the localisation argument we need for Theorem \ref{pf_moments_meso} in Section \ref{chain_arg}. Theorems \ref{pf_moments_constant} and \ref{pf_moments_meso} will be proven in Section \ref{pfm_proof}. We defer the proof of Theorem \ref{short_range_max} to Section \ref{typical_max_section}. For a more general review of results relating to the progress on the FHK conjectures and critical multiplicative chaos, see the papers of Arguin \cite{arguin16}, Bailey--Keating \cite{bk22} and Powell \cite{powell22} respectively.
\subsection*{Acknowledgements}
The author would like to thank his supervisors Oleksiy Klurman and Joseph Najnudel for reading an earlier draft of this paper. He also thanks Jad Hamdan in particular for sharing a version of \cite{ah24} early. The author also thanks Emma Bailey and Andrew Pearce-Crump for their encouragement on this problem. Some of this work was completed while the author was visiting the Centre de recherches de math\'{e}matiques de Montr\'{e}al. He thanks them for excellent working conditions. The author is supported by the Heilbronn Institute of Mathematical Research.
\subsection*{AI usage declaration}
The author acknowledges the use of ChatGPT 6 Astra for discussions and proofreading surrounding this paper. All mathematical arguments developed with this assistance have been reworked and validated by hand and the author takes full responsibility for the accuracy and rigour of the proofs.
\section{Preliminary Results} \label{sec2}
Here we collect the results required for the partition of unity and on the number theory side.
\subsection{A partition of unity}
The following partition of unity will allow us to transfer between estimates over character sums and estimates for sums of Steinhaus random multiplicative functions while preserving barrier conditions.
\begin{lemma}
    (Approximation Result 1 in \cite{harper2023typicalsizecharacterzeta}) Let $N$ be large and $\delta>0$ be small. Then there exist functions $g,g_{N+1}:\mathbb{R} \to \mathbb{R}$ (depending on $N$ and $\delta$) such that for $g_j(x) = g(x-j)$ where $j$ are all the integers such that $|j| \leq N$, the following hold:
    \begin{enumerate}
        \item $\sum_{|j| \leq N}g_j(x) + g_{N+1}(x) = 1, \ \forall x \in \mathbb{R}$;
        \item $g(x) \geq 0, \ \forall x \in \mathbb{R}$ and $g(x) \leq \delta$ whenever $|x| \geq 1$;
        \item $g_{N+1}(x) \geq 0, \ \forall x \in \mathbb{R}$ and $g_{N+1}(x) \leq \delta$ whenever $|x| \leq N$;
        \item $\forall \ell \in \mathbb{N}$ and $\forall x \in \mathbb{R}$, the derivative estimate $|g^{(\ell)}(x)| \leq \frac{1}{\pi(\ell+1)}(\frac{2\pi}{\delta})^{\ell+1}$ holds.
    \end{enumerate}
    \label{partition_result}
\end{lemma}
Note that from the construction of the partition of unity, we get the stronger estimate with $|x| > N+1$
$$\sum_{|j| \le N}g_j(x)  \le \frac{\delta}{1+\delta} \le \delta.$$
Such functions are constructed in \cite{harper2023typicalsizecharacterzeta}.
\subsection{Estimates involving character averaging}
We need several mixed and high moment estimates in this proof. The first is a Gaussian moment estimate for character sums over primes and their squares.
\begin{lemma}
    (Lemma 1 in \cite{harper2023typicalsizecharacterzeta}) Let $x \geq 1$ and $(c(n))$ be a set of arbitrary complex numbers. Let $\mathcal{P}$ be a finite set of primes, and let $\mathcal{R}$ be the set of the elements of $\mathcal{P}$ and their squares. Let $U = \max\{r \in \mathcal{R}\}$. Finally, set $R(\chi) = \sum_{r \in \mathcal{R}}\frac{a(r)\chi(r)}{\sqrt{r}}$ with $(a(r))$ a separate set of arbitrary complex numbers. Then for any $k \in \mathbb{N}$ such that $xU^k < q$
    $$\mathbb{E}^{\text{char}}\big|\sum_{n \leq x}c(n)\chi(n)\big|^2|R(\chi)|^{2k} \ll \big(\sum_{n \leq x}\tilde{d}(n)|c(n)|^2\big)\cdot(k!)\big(\sum_{r \in \mathcal{R}}\frac{2v_r|a(r)|^2}{r}\big)^k$$
    where $\tilde{d}(n) := \sum_{d|n}\mathbf{1}_{p|d \Rightarrow p \in \mathcal{P}}$ and $v_r$ is $1$ if $r$ is a prime, and $6$ otherwise. With no character sum, we obtain the sharper estimate when $U^k < q$
    $$\mathbb{E}^{\text{char}}[|R(\chi)|^{2k}] \ll (k!)(\sum_{r \in \mathcal{R}}\frac{v_r|a(r)|^2}{r})^k.$$
    \label{even_moment}
\end{lemma}
The key result is the following randomisation estimate which allows us to pass from character averages to averages over Steinhaus multiplicative functions.
\begin{lemma}
     (Mixed character randomisation, Proposition 1 \cite{harper2023typicalsizecharacterzeta}) Suppose $Y \in \mathbb{N}$. Then for each $1 \le i \le Y$, let $g_{j}^{(i)}$ with associated parameters $(N_i,\delta_i)$ be as described in Lemma \ref{partition_result}, $N_i \in \mathbb{N}$ with $N_M = \max_i\{N_i\}, N_m=\min_i\{N_i\}$ and $\delta_M =\max_i\{\delta_i\}$. Set $\mathcal{J}_i = [-N_i,N_i+1] \cap \mathbb{Z}$ for each $i$.  Further, let $x \geq 1$ and suppose $(c(n))_{n \leq x}$ are arbitrary complex coefficients such that $|c(n)| \leq 1$ for each $n$. Finally, suppose the $P_i$ are large with $P_m = \min_i\{P_i\}, \ P_M =\max_i\{P_i\}$ and $Y \in \mathbb{N}$ such that \begin{align*}
         x\exp(400Y\sum_{i=1}^Y(1/\delta_i)^2\log P_i\log (N_M\log  P_i)) < q.
     \end{align*}
    Then for $(j(1),...,j(Y)) \in \prod_{i \le Y}\mathcal{J}_i$ and sequences $((a_i(p),a_i(p^2))_{p \leq P_i})_{i \leq Y}$ of complex numbers with absolute value $\leq 1$, we have
    \begin{align*}
        &\mathbb{E}^{\text{char}}\prod_{i=1}^Yg^{(i)}_{j(i)}\big(\Re\Big[\sum_{p \leq P_i}\big(\frac{a_i(p)\chi(p)}{p^{1/2}} + \frac{a_i(p^2)\chi(p^2)}{p}\big)\Big]\big)\big|\sum_{n \leq x}\frac{c(n)\chi(n)}{n^{1/2}}\big|^2 \\
        =& \mathbb{E}^f\prod_{i=1}^Yg_{j(i)}^{(i)}\big(\Re\Big[\sum_{p \leq P_i}\big(\frac{a_i(p)f(p)}{p^{1/2}} + \frac{a_i(p^2)f(p^2)}{p}\big)\Big]\big)\big|\sum_{n \leq x}\frac{c(n)f(n)}{n^{1/2}}\big|^2 + O\big(\frac{\log( 2x)\log P_M}{(N_m\log P_m)^{Y/\delta_M^2}}\big).
    \end{align*} 
    Analogous statements connecting $\mathbb{E}^{\text{cont}}$ and $\mathbb{E}^f$ hold when one replaces $q$ with $T$, $\chi(p),\chi(p^2)$ with $p^{-it},p^{-2it}$ and $\mathbb{E}^{\text{char}}$ with $\mathbb{E}^{\text{cont}}$.
    \label{randomisation}
\end{lemma}
\begin{proof}
This is identical to that of Proposition 1 in \cite{harper2023typicalsizecharacterzeta} with some minor changes, and the $t$-aspect version holds if we use Lemma 2 from \cite{harper2019partition} in place of Lemma \ref{even_moment} (see Section 4.1 in \cite{harper2023typicalsizecharacterzeta}). The main identification comes from rewriting the condition $xP^{4SY } < q$ so it reflects the size of the different $P_i$. We also compute the remainder terms of the random side since our version of the proposition allows for different majorants $g_j^{(i)}$ simultaneously, so the length conditions are less troublesome. The polynomial majorant of the remainder term has the same expectation on both the character and random sides. We have the following recursion for the relevant remainder terms when expanding the Taylor remainder, which is
\begin{equation*}\begin{split}
    \ll \mathbb{E}^{f}&\bigg[\sum_{i=1}^Y\frac{N_i|2\pi/\delta_i|^{2S_i}}{\delta_i S_i (2S_i)!}\bigg|\sum_{p \le P_i}\Big(\frac{a_i(p)f(p)}{p^{1/2}}+\frac{a_i(p^2)f(p^2)}{p}\Big)\bigg|^{2S_i} \\
    &\cdot \prod_{\ell=1}^{i-1}\bigg(1+O\bigg(\frac{N_\ell|2\pi/\delta_\ell|^{2S_\ell}}{\delta_\ell S_\ell (2S_\ell)!}\Big|\sum_{p \le P_\ell}\Big(\frac{a_\ell(p)f(p)}{p^{1/2}}+\frac{a_\ell(p^2)f(p^2)}{p}\Big)\Big|^{2S_\ell}\bigg)\bigg)\bigg|\sum_{n \le x}\frac{c(n)f(n)}{\sqrt{n}}\bigg|^2\bigg]
\end{split}
\end{equation*}
At this point, we then can exchange the order of summation and expectation, and then apply weighted H\"older's inequality over the product to obtain the above is
\begin{equation*}\begin{split}
    \ll &\sum_{i=1}^Y\frac{N_i|2\pi/\delta_i|^{2S_i}}{\delta_i S_i (2S_i)!}\mathbb{E}^f\bigg[\bigg|\sum_{p \le P_i}\Big(\frac{a_i(p)f(p)}{p^{1/2}}+\frac{a_i(p^2)f(p^2)}{p}\Big)\bigg|^{2S_i} \\
    &\cdot \prod_{\ell=1}^{i-1}\bigg(1+O\bigg(\frac{N_\ell|2\pi/\delta_\ell|^{2S_\ell}}{\delta_\ell S_\ell (2S_\ell)!}\Big|\sum_{p \le P_\ell}\Big(\frac{a_\ell(p)f(p)}{p^{1/2}}+\frac{a_\ell(p^2)f(p^2)}{p}\Big)\Big|^{2S_\ell}\bigg)\bigg)\bigg|\sum_{n \le x}\frac{c(n)f(n)}{\sqrt{n}}\bigg|^2\bigg] \\
    \ll &\sum_{i=1}^Y\frac{N_i|2\pi/\delta_i|^{2S_i}}{\delta_i S_i (2S_i)!}(\mathbb{E}^f\bigg[ \bigg|\sum_{p \le P_i}\Big(\frac{a_i(p)f(p)}{p^{1/2}}+\frac{a_i(p^2)f(p^2)}{p}\Big)\bigg|^{2iS_i}\bigg|\sum_{n \le x}\frac{c(n)f(n)}{\sqrt{n}}\bigg|^2)^{1/i} \\
    &\cdot \prod_{\ell=1}^{i-1}\bigg[(\mathbb{E}^f\bigg[\bigg(\frac{C_\ell N_\ell|2\pi/\delta_\ell|^{2S_\ell}}{\delta_\ell S_\ell (2S_\ell)!}\Big|\sum_{p \le P_\ell}\Big(\frac{a_\ell(p)f(p)}{p^{1/2}}+\frac{a_\ell(p^2)f(p^2)}{p}\Big)\Big|^{2S_\ell}\bigg)^{i}\bigg|\sum_{n \le x}\frac{c(n)f(n)}{\sqrt{n}}\bigg|^2\bigg])^{1/i}
    \\ & \quad + (\mathbb{E}^f\bigg[\bigg|\sum_{n \le x}\frac{c(n)f(n)}{\sqrt{n}}\bigg|^2\bigg])^{1/i}\bigg]
\end{split}
\end{equation*}
where the final step holds by weighted H\"older and then Minkowski's inequality. The resulting moments can be handled by a random version of Lemma \ref{even_moment} in a similar way to Harper using Stirling approximation to handle the factorials. This leads to the choice $S_i = 100Y(1/\delta_i)^2\log(N_i\log P_i)$, which has the twist condition
$$x\exp\Big(400Y\sum_{i=1}^Y(1/\delta_i)^2\log P_i\log(N_i\log P_i)\Big) \le x\exp\Big(400Y\sum_{i=1}^Y(1/\delta_i)^2\log P_i\log(N_M\log P_i)\Big) < q$$
which is more than sufficient for our theorem. Further, we can estimate $\sum_{n \leq x}\frac{|c(n)|^2\tilde{d}(n)}{n}$ in a sharper manner. Indeed, we observe that by Lemma \ref{even_moment}
$$\sum_{n \leq x}\frac{|c(n)|^2\tilde{d}(n)}{n} \leq \sum_{\substack{m \leq x \\ P^+(m) \leq P_M}}\frac{d(m)}{m}\sum_{\substack{n \leq x/m \\ P^-(n) > P_M}}\frac{1}{n} \ll \log  (2x)\log  P_M.$$
This leads to the remainder being of size at most 
$$\log(2x)\log(P_M)\bigg[\prod_{i=1}^Y\left(1+\frac{N_i}{\delta_i S_i}(\frac{2e\pi^2}{100})^{S_i}\right)-1\bigg] \ll \frac{\log (2x) \log P_M}{(N_m \log(P_m))^{Y/\delta_M^2}}$$
after expanding the product as an exponential.
\end{proof}
Note that virtually all of the choices of $\delta$ and $P$ are motivated by the condition relating the various parameters above, so it is in our interest to make sure both $P$ and $\delta$ are as large as possible while getting this result to hold. This is where our restriction on the size of $\theta$ comes in since the arithmetic limitations of our approach only are restrictive when $\theta$ approaches $-1/2$ from above too quickly. We introduce the approximate functional equation for primitive Dirichlet characters which will be essential throughout.
\begin{lemma}
    (Theorem 5.3 and Proposition 5.4 in \cite{IK04}) Let $q$ be a large prime, and $\chi$ be a primitive Dirichlet character modulo $q$. Then for any $|h| \leq 1$
    $$L(1/2+ih,\chi) = \sum_{n \le q/2}\frac{\chi(n)V_{1/2+ih}(n/\sqrt{q})}{n^{1/2+ih}} + \psi(1/2+ih,\chi)\sum_{n \le q/2}\frac{\overline{\chi(n)}V_{1/2-ih}(n/\sqrt{q})}{n^{1/2-ih}} + O(q^{-1}),$$
    where $|\psi(1/2+it,\chi)|=1$ for any non-principal $\chi$ modulo $q$ with smoothing function $V_s$ which has the following properties:
    \begin{align*}
        V_s(u) = \begin{cases}
            1+O(u^{1/3}) \qquad &u \in [0,1]\\
            O(u^{-A}) \qquad &u\ge 1.
        \end{cases}
    \end{align*}
    for any $A>0$. Set
    $$R_\chi(h) = L(1/2+ih,\chi) - \sum_{n \le \sqrt{q/2\pi}}\frac{\chi(n)}{n^{1/2+ih}}-\psi(1/2+ih,\chi)\sum_{n \le \sqrt{q/2\pi}}\frac{\overline{\chi(n)}}{n^{1/2-ih}}$$
    Consequently, we have for $I \subseteq [-1,1]$ non-empty
    $$\mathbb{E}^{\text{char}}[|R_\chi(h)|^2] \ll 1, \qquad \mathbb{E}^{\text{char}}[|\partial_hR_\chi(h)|^2] \ll \log^2q, \qquad \mathbb{E}^{\text{char}}[\sup_{h \in I}|R_\chi(h)|^2] \ll 1+|I|\log q.$$
\label{Dirichlet_AFE}
\end{lemma}
\begin{proof}
    It is sufficient to prove the last claims. By symmetry, this reduces to proving
    $$\mathbb{E}^{\text{char}}\bigg[\Big|\sum_{n \le q/2}\frac{\chi(n)V(n/\sqrt{q})}{n^{1/2+ih}}-\sum_{n \le \sqrt{q/2\pi}}\frac{\chi(n)}{n^{1/2+ih}}\Big|^2\bigg]$$
    by the triangle inequality for order of magnitude bounds. Then by character orthogonality, one has this is
    $$\ll \sum_{n}\frac{|V(n/\sqrt{q})-\mathbf{1}_{n \le \sqrt{q/2\pi}}|^2}{n} \ll q^{-1/3}\sum_{n \le \sqrt{q/2\pi}}n^{-1/3} + \sum_{\sqrt{q/2\pi} < n \le \sqrt{q}}n^{-1} +q^{A}\sum_{\sqrt{q} <n \le q/2}n^{-1-2A}$$
    is $\ll 1$. As for the derivative, we have that taking the derivative with respect to $s$, results in expressions of the form
    $$-i\sum_{n \le q/2} \frac{\chi(n)(V_{s}(n/\sqrt{q})-\mathbf{1}_{n \le \sqrt{q/2\pi}}(n))\log n}{n^{1/2+ih}} + \sum_{n \le q/2}\frac{\chi(n)\partial_sV_s(n/\sqrt{q})}{n^{1/2+ih}}.$$
    The smoothing functions $V$ have decay $\partial_sV_s(x) \ll \min\{x^{1/3},x^{-A}\}$ for any $A>0,x\ge 0$ and $\partial_h\psi \ll \log q$ as well. Repeating the decomposition as above we have
    $$\mathbb{E}[|\partial_hR_\chi(h)|^2] \ll \log^2 q.$$
    As for the supremum bound, by the fundamental theorem of calculus and Cauchy Schwarz
    \begin{align*}
        \mathbb{E}^{\text{char}}[\sup_{h \in I} |R_\chi(h)|^2] \ll& \frac{1}{|I|}\int_I\mathbb{E}^{\text{char}}[|R_\chi(h)|^2]dh \\ \quad &+ \Big(\int_I\mathbb{E}^{\text{char}}[|R_\chi(h)|^2]dh\Big)^{1/2}\Big(\int_I\mathbb{E}^{\text{char}}[|\partial_hR_\chi(h)|^2]dh\Big)^{1/2}.
    \end{align*}
    From the previous estimates, our conclusion follows immediately.
\end{proof}
Also note that the principal character modulo $q$ is never an issue for us since we have that
$$L(1/2+ih,\chi_0) = (1-q^{-1/2-ih})\zeta(1/2+ih),$$
and $\zeta$ is bounded for $|h| \le 1$. With this in hand, we want to introduce a suitable approximation when we average over all characters. We do this to shorten the Dirichlet polynomial in the above. Both of the following lemmas hold with Steinhaus random multiplicative functions $f$ in place of $\chi$, and help set up the independence of disjoint prime ranges on the random side. As these proofs are very similar, we omit the second.
\begin{lemma}
    Fix $\theta \in (-1/2,0]$. Then uniformly for all $|\sigma-1/2| \leq \frac{1}{\log  q}$ and $|h| \leq \log ^{\theta} q$
    \begin{equation}
        \mathbb{E}^{\text{char}}\big|\sum_{n \leq \sqrt{q/2\pi}}\frac{\chi(n)}{n^{\sigma+ih}}-\sum_{\substack{m \leq q^{\varepsilon_\theta} \\ P^+(m) \leq P_\theta}}\frac{\chi(m)}{m^{1/2+ih}}\sum_{\substack{n \leq q^{1/2-2\varepsilon_\theta} \\ P^-(n) > P_\theta}}\frac{\chi(n)}{n^{\sigma+ih}}\big|^2 \ll \frac{\log ^{1+\theta} q}{(\log \log  q)^2}.  \label{factored_eqn}
    \end{equation} \label{factored_form}
\end{lemma}
\begin{proof}
    We use the triangle inequality to handle the lack of shift in the smooth part of the sum. This means the left hand side of equation (\ref{factored_eqn}) satisfies
    \begin{align*}
        \ll& \mathbb{E}^{\text{char}}\big|\sum_{n \leq \sqrt{q/2\pi}}\frac{\chi(n)}{n^{\sigma+ih}}-\sum_{\substack{m \leq q^{\varepsilon_\theta} \\ P^+(m) \leq P_\theta}}\frac{\chi(m)}{m^{\sigma+ih}}\sum_{\substack{n \leq q^{1/2-2\varepsilon_\theta} \\ P^-(n) > P_\theta}}\frac{\chi(n)}{n^{\sigma+ih}}\big|^2 \\
        &+ \mathbb{E}^{\text{char}}\big|\sum_{\substack{n \leq q^{1/2-2\varepsilon_\theta} \\ P^-(n) > P_\theta}}\frac{\chi(n)}{n^{\sigma+ih}}\big|^2\mathbb{E}^{\text{char}}\big|\sum_{\substack{m \leq q^{\varepsilon_\theta} \\ P^+(m) \leq P_\theta}}\frac{\chi(m)}{m^{\sigma+ih}}-\sum_{\substack{m \leq q^{\varepsilon_\theta} \\ P^+(m) \leq P_\theta}}\frac{\chi(m)}{m^{1/2+ih}}\big|^2.
    \end{align*}
    We achieve the second term by using the orthogonality of Dirichlet characters. For the first term, we are counting the numbers that are between $q^{1/2-2\varepsilon_\theta}$ and numbers with a $P$-smooth part which is greater than $q^{\varepsilon_\theta}$. Then by orthogonality, the first term is bounded by
    \begin{align*}
        \sum_{q^{1/2-2\varepsilon_\theta} \leq n \leq q^{1/2}}  \frac{1}{n^{2\sigma}}+\sum_{\substack{q^{\varepsilon_\theta}< m \leq q^{1/2-2\varepsilon_\theta} \\ P^+(n) \leq P_\theta}}\frac{1}{m^{2\sigma}}\sum_{\substack{n \leq q^{1/2-2\varepsilon_\theta} \\ P^-(n) > P_\theta}}\frac{1}{n^{2\sigma}} 
        \ll& \varepsilon_\theta\log  q + \log^{1+\theta}  q\sum_{\substack{ m > q^{\varepsilon_\theta} \\ P^+(n) \leq P_\theta}}\frac{1}{m^{2\sigma}} \\ \ll& \frac{\log ^{1+\theta} q}{(\log \log  q)^2} + \log  q\exp\big(-\frac{\log ^{1+2\theta} q}{(\log \log  q)^2}\big).
    \end{align*}
    The final line is achieved by using Rankin's trick to bound the contribution of the smooth numbers by $ e^{-\frac{\varepsilon_\theta\log  q}{\log P_\theta}}\log P_\theta$ and by noting $\theta > -1/2$. As for the second term, this reduces to bounding
    $$\log^{1+\theta}q\sum_{\substack{m \leq q^{\varepsilon_\theta} \\ P^+(m) \leq P_\theta}}\frac{|m^{\sigma-1/2}-1|^2}{m}.$$
    Observing that $|\sigma-1/2| \leq \frac{1}{\log  q}$, then $|e^{\log m|\sigma-1/2|}-1| \leq \frac{\log m}{\log ^{1+\theta} q}$ by Taylor expansion, so the contribution of the second term is $\ll \varepsilon_\theta\log ^{1+\theta} q$.
\end{proof}
\begin{lemma}
    Fix $\theta \in (-1/2,0]$. Suppose $P > P_\theta$ and $\min\{\frac{1}{\log  q},\varepsilon_\theta\} < \varepsilon_1 \leq \frac{1}{10}$ say, then we have uniformly for $|\sigma-1/2| \leq \frac{1}{\log  q}$ and $|h| \leq \log ^{\theta} q$ with $\varepsilon=\varepsilon_1+\varepsilon_\theta$
    \begin{equation*}
    \mathbb{E}^{\text{char}}\big|\sum_{\substack{n \leq q^{1/2-2\varepsilon_\theta} \\ P^{-}(n) >P_\theta}}\frac{\chi(n)}{n^{\sigma+ih}}-\sum_{\substack{m \leq q^{\varepsilon} \\ p|m \Rightarrow P_\theta < p \leq P}}\frac{\chi(m)}{m^{1/2+ih}}\sum_{\substack{n \leq q^{1/2-2\varepsilon} \\ P^{-}(n) > P}}\frac{\chi(n)}{n^{\sigma+ih}}\big|^2 \ll \log ^{1+\theta} q(\varepsilon+e^{-\frac{\varepsilon\log  q}{\log P}}).
    \end{equation*}
    \label{logcor_factored_form}
\end{lemma}
The following conditional fourth moment estimate is only required for Theorem \ref{short_range_max}. We let
$$\mathcal{H}_\chi(h;V) := \big\{\big|\sum_{\substack{m \leq q^\varepsilon \\ P^+(m) \leq P}}\frac{\chi(m)}{m^{1/2+ih}}\big| > V^{-1}\log P\big\}.$$
\begin{lemma}
    Uniformly for large $ P \leq \sqrt{q}, \ 0 <\varepsilon < 1/10$ with $V>0$ and $|\sigma-1/2| \leq 1/\log  q$, we have
    $$\mathbb{E}^{\text{char}}\big[\mathbf{1}_{\mathcal{H}_\chi(h;V) \ \text{fails}}(h)\big|\sum_{\substack{m \leq q^\varepsilon \\ P^+(m) \leq P}}\frac{\chi(m)}{m^{1/2+ih}}\sum_{\substack{n \leq q^{1/2-2\varepsilon} \\ P^{-}(n) > P}}\frac{\chi(n)}{n^{\sigma+ih}}\big|^4\big] \ll \frac{1}{V^2}\frac{\log ^4q}{\log P}.$$ \label{conditional_fourth}
\end{lemma}
\begin{proof}
    Using the conditioning given by the failure of the event $\mathcal{H}_\chi(h;V)$, we can upper bound the left hand side by
    \begin{align*}
        &\big(\frac{\log P}{V}\big)^2\mathbb{E}^{\text{char}}\big|\sum_{\substack{m \leq q^\varepsilon \\ P^+(m) \leq P}}\frac{\chi(m)}{m^{1/2+ih}}\big|^2\big|\sum_{\substack{n \leq q^{1/2-2\varepsilon} \\ P^{-}(n) > P}}\frac{\chi(n)}{n^{\sigma+ih}}\big|^4 \\
        \ll& \big(\frac{\log P}{V}\big)^2\mathbb{E}^{\text{char}}\big|\sum_{\substack{m \leq q^\varepsilon \\ P^+(m) \leq P}}\frac{\chi(m)}{m^{1/2+ih}}\big|^2\sum_{\substack{n \leq q^{1-4\varepsilon} \\ P^{-}(n) > P}}\frac{|c(n)|^2}{n^{2\sigma}},
    \end{align*}
    where $c(n) := \#\{(d_1,d_2): d_1,d_2 \leq q^{1/2-2\varepsilon},  n=d_1d_2\}$ (where we have applied Lemma \ref{even_moment} to the rough sum due to orthogonality of the characters). Then Lemma \ref{even_moment} implies that this is
    $$\ll \Big(\frac{\log P}{V}\Big)^2\sum_{P^+(m) \leq P}\frac{1}{m}\sum_{\substack{n \leq q^{1-4\varepsilon} \\ P^-(n) > P}}\frac{(c(n))^2}{n} \ll \frac{1}{V^2}\log ^3 P\sum_{\substack{n \leq q^{1-4\varepsilon} \\ P^-(n) > P}}\frac{d_4(n)}{n}$$
    where $d_k(n) = \#\{(d_1,...,d_k):\prod_{i=1}^kd_i=n\}$ is the standard $k$-fold divisor function (we used that $(d_2(n))^2 \leq d_4(n)$ here). The final term is bounded by $$\prod_{P < p \leq q^{1-4\varepsilon}}\big(1-\frac{4}{p}\big)^{-1} \ll \big(\frac{\log  q}{\log P}\big)^4,$$ which yields the claim.  
\end{proof}
\section{The lognormal contribution conditioning} \label{chain_arg}
In this section we will prove that for a short enough Dirichlet polynomial supported on primes and their squares, once randomised we can relate it to an Euler product.
\begin{lemma}
    Let $\theta \in (-1/2,0]$. Then uniformly for all $|h| \leq \log ^{\theta} q/2$
    \begin{equation}
        \mathbb{E}\big|\sum_{\substack{m \leq q^{\varepsilon_\theta} \\ P^+(m) \leq P_\theta}}\frac{f(m)}{m^{1/2+ih}} - F_{P_\theta}(1/2+ih)\big|^2\ll e^{-\varepsilon_\theta\log^{1+\theta}q}\log^{|\theta|} q  \label{exp_log_eq}
    \end{equation} 
    \label{exp_logn}
\end{lemma}
\begin{proof}
    We see that
    \begin{equation*}
        \begin{split}\mathbb{E}\big|\sum_{\substack{m \leq q^{\varepsilon_\theta} \\ P^+(m) \leq P_\theta}}\frac{f(m)}{m^{1/2+ih}} - F_{P_\theta}(1/2+ih)\big|^2& \ll \sum_{\substack{m > q^{\varepsilon_\theta} \\ P^+(m) \le P_\theta}}\frac{1}{m}\ll e^{-\varepsilon_\theta\log^{1+\theta}q }\log^{|\theta|}q,\end{split}
    \end{equation*}
    which follows by Rankin's trick in the same way as Lemma \ref{factored_form}
\end{proof}
Combining Lemma \ref{randomisation} with Lemma \ref{exp_logn} then yields the following proposition. This is what allows us to condition on one point per mesoscopic window.
\begin{proposition}\label{chaining}
    (Single point conditioning) Fix $\theta \in (-1/3,0]$ and $W \ge W_0$. Then uniformly
    \begin{align*}
        \int_{|h| \leq \log ^\theta q/2}\mathbb{E}_W^{\text{char}}|\sum_{\substack{m \leq q^{\varepsilon_\theta} \\ P^+(m) \leq P_\theta}}\frac{\chi(m)}{m^{1/2+ih}}|^2dh \ll e^{2W}\log ^{\theta} q
    \end{align*} 
    An analogous statement holds for $\mathbb{E}^f$ and Steinhaus $f$.
\end{proposition}
    \begin{proof}
    Since our integral is over a finite length, we can approximate $\mathbf{1}_{Y_{P_\theta,\chi(0) \in [W,W+1]}}(0)$ by a smooth function $g_0(Y_{P_\theta,\chi}(0))$ which satisfies various properties stated in Lemma \ref{partition_result}. At this point, we have shown that for some $\delta_{P_\theta},N_{P_\theta}$ to be chosen later that
    \begin{align*}
        &\int_{|h| \leq \log ^{\theta} q/2}\mathbb{E}^{\text{char}}_W|\sum_{\substack{m \leq q^{\varepsilon_\theta} \\ P^+(m) \leq P_\theta}}\frac{\chi(m)}{m^{1/2+ih}}|^2dh \\ \leq& \int_{|h| \leq \log ^{\theta} q/2}\mathbb{E}^{\text{char}}\big[g_0(Y_{P_\theta,\chi}(0))|\sum_{\substack{m \leq q^{\varepsilon_\theta} \\ P^+(m) \leq P_\theta}}\frac{\chi(m)}{m^{1/2+ih}}|^2\big]dh.
    \end{align*}
    At this point, we can apply Lemma \ref{randomisation} to allow us to switch to a statement involving random multiplicative functions. This is very important because it allows us to compare this Dirichlet polynomial to an Euler product as there are no periodicity conditions to worry about. Then this is
    \begin{align*}
        &\int_{|h| \leq \log ^{\theta} q/2}\mathbb{E}^{\text{char}}\big[g_0(Y_{P_\theta,\chi}(0))|\sum_{\substack{m \leq q^{\varepsilon_\theta} \\ P^+(m) \leq P_\theta}}\frac{\chi(m)}{m^{1/2+ih}}|^2\big]dh \\=&\int_{|h| \leq \log ^{\theta} q/2}\mathbb{E}^f\big[g_0(Y_{P_\theta,f}(0))|\sum_{\substack{m \leq q^{\varepsilon_\theta} \\ P^+(m) \leq P_\theta}}\frac{f(m)}{m^{1/2+ih}}|^2\big]dh + O\big(\frac{\log P_\theta}{(N_{P_\theta}\log P_\theta)^{1/\delta_{P_\theta}^2}}\big).
    \end{align*}
    Now we can undo the smoothing, noting that $g_0(Y_{P_\theta,f}(0)) \le C(\mathbf{1}_{Y_{P_\theta,f}(0) \in [W-1,W+2]}(0) + \delta_{P_\theta})$. Consequently, we have
    \begin{align*}
        &\int_{|h| \leq \log ^{\theta} q/2}\mathbb{E}^f\big[g_0(Y_{P_\theta,f}(0))|\sum_{\substack{m \leq q^{\varepsilon_\theta} \\ P^+(m) \leq P_\theta}}\frac{f(m)}{m^{1/2+ih}}|^2\big]dh \\
        \ll& \int_{|h| \leq \log ^{\theta} q/2}\mathbb{E}^f\big((\mathbf{1}_{Y_{P_\theta,f}(0) \in [W-1,W+2]}(0)+\delta_{P_\theta})|\sum_{\substack{m \leq q^{\varepsilon_\theta} \\ P^+(m) \leq P_\theta}}\frac{f(m)}{m^{1/2+ih}}|^2\big)dh.
    \end{align*}
    At this point, we can apply the triangle inequality, and obtain the following upper bound on the above
    \begin{equation}\label{pen_step}\begin{split}
        \ll& \int_{|h| \leq \log ^\theta(q)/2}\mathbb{E}^f\big[\big|\sum_{\substack{m \leq q^{\varepsilon_\theta} \\ P^+(m) \leq P_\theta}}\frac{f(m)}{m^{1/2+ih}}-F_{P_\theta}(1/2+ih)\big|^2\big]dh \\
        &+\int_{|h| \leq \log ^\theta(q)/2}\mathbb{E}^f\big[(\mathbf{1}_{Y_{P_\theta,f}(0) \in [W-1,W+2]}(0)+\delta_{P_\theta})|F_{P_\theta}(1/2+ih)|^2\big]dh.
    \end{split}
    \end{equation}
    From Lemma \ref{exp_logn}, the first term's contribution is negligible. This leaves us to handle the term involving the exponential. We have $$\delta_{P_\theta}\int_{|h| \leq \log ^\theta(q)/2}\mathbb{E}^f[|F_{P_\theta}(1/2+ih)|^2]dh\ll \delta_{P_\theta}$$ by orthogonality. The second term in equation (\ref{pen_step}) is
    \begin{equation}
        \ll \int_{|h| \leq \log ^\theta(q)/2}\mathbb{E}_f\big[\mathbf{1}_{Y_{P_\theta,f}(0) \in [W-1,W+2]} |\Delta_f(h)|^2e^{2Y_{P_\theta,f}(0)}\big]dh,\label{final_form}
    \end{equation}
    where $\Delta_f(h) =F_{P_\theta}(1/2+ih)e^{-Y_{P_\theta,f}(0)}$. By the conditioning, we can immediately evaluate the $e^{2Y_{P_\theta,f}(0)}$, so it remains to calculate $\mathbb{E}^f[|\Delta_f(h)|^2]$. By definition, we have
    $$|\Delta_f(h)| = \exp \big(\sum_{p \le P_\theta}\Re\big(\frac{f(p)(p^{-ih}-1)}{p^{1/2}} + \frac{f(p^2)(p^{-2ih}-1)}{2p}+O(p^{-3/2}\big)\big),$$
    and on our interval, $|h| \ll \log^{\theta} q$. Therefore, we have $\mathbb{E}^f|\Delta_f(h)|^2 \ll 1$, so equation (\ref{final_form}) is bounded by
    $\ll e^{2W}\log ^{\theta} q$.
    We also see that our parameter choices satisfy
    $$q^{\varepsilon_\theta} P_\theta^{800\max\{\log\log q,\log^{2|\theta|}q\} (\log\log q)^5} =q^{o(1)}$$
    with our choices dependent on $\theta >-1/3$ with
    $$\delta_{P_\theta} = (\log\log q)^{-2}\min\{\log^\theta q, (\log\log q)^{-1}\}, \ N_{P_\theta} = 1.$$
    \end{proof}
\section{Proof of Theorems \ref{pf_moments_constant} and \ref{pf_moments_meso}} \label{pfm_proof}
As we have already proved Proposition \ref{chaining}, we can largely follow Harper's treatment of the problem for the log-correlated regime. Once we have handled this contribution, we then are free to extract the log-normal contribution. Throughout this section we choose $P=q^{1/(\log\log q)^8}$ and $\varepsilon=1/(\log\log q)^2+\varepsilon_\theta$.\par
For this purpose, fix $\theta \in (-1/3,0]$. We define the set of points for each $0 \leq j \leq \log \log P+\theta\log \log  q-1$
\begin{equation}
    h(j) := \max\big\{u \leq h(j-1):u=\frac{n}{(\log P/e^j)\log(P_q(\theta)/e^j)} \text{ for some } n \in \mathbb{Z}\big\} \label{discrete_h}
\end{equation}
with $h(-1) = h$ and $P_q(\theta) = \log P\log^\theta q$. We then define the partial Euler products
\begin{equation}
    I_{\ell,\chi}(h) := \prod_{P^{e^{-(\ell+1)}} < p \leq P^{e^{-\ell}}}\big(1-\frac{\chi(p)}{p^{1/2+ih}}\big)^{-1}, \quad  I_{\ell,f}(h) := \prod_{P^{e^{-(\ell+1)}} < p \leq P^{e^{-\ell}}}\big(1-\frac{f(p)}{p^{1/2+ih}}\big)^{-1}.\label{char_increment}
\end{equation}
Next, we define the functions $P_q(\theta) = \log P\log^\theta q$ and $t(j) = C\min\{\sqrt{\log P_q(\theta)},\frac{1}{1-r}\} + 2\log \log (P_q(\theta)/e^j)$ where $C>0$ is a large absolute constant. Finally, we set $\mathcal{G}_\chi$ as the event such that for all $|h| \leq \log^\theta q/2$ and all $0 \leq j \leq \log P_q(\theta) - B -1$, we have
$$\big|\sum_{\ell=j}^{\lfloor\log P_q(\theta)\rfloor-B-1}\log |I_{\ell,\chi}(h(\ell))|\big|\leq \log P_q(\theta) - j +t(j).$$
Note this is defined in exactly the same way as the event $\mathcal{G}_t$ in \cite{harper2019partition}, only we have $\chi$ in place of $t$ (with $B$ being a parameter required due to the discussion in Section \ref{prob_heur}). We define $\mathcal{G}_{f}$ in almost the same way, with $f(p)$ a sequence of independent Steinhaus random variables being used in place of Dirichlet character $\chi$ modulo $q$, only where we use the function $t_{f}(j) = t(j) + 19/3$ (the $19/3$ is used to absorb issues arising from extra convergent terms). We denote the event $\mathcal{G}_\chi(h)$ (respectively $\mathcal{G}_f(h)$) if $\mathcal{G}_\chi$ occurs at a specific $h$ rather than all $h$. We then have the following propositions, from which we may deduce Theorems \ref{pf_moments_constant} and \ref{pf_moments_meso}.
\begin{proposition}
    Fix $\theta \in (-1/3,0]$. Then uniformly in $q$ sufficiently large in terms of $C$ and $2/3 \leq r \leq 1$
    $$\mathbb{P}_q\big(\mathcal{G}_\chi \text{ fails}\big) \ll e^{-2C\min\{\sqrt{\log P_q(\theta)},(1-r)^{-1}\}}.$$ \label{bad_chaining_moment}
\end{proposition}
\begin{proposition}
    Fix $\theta \in (-1/3,0]$. Then uniformly for sufficiently large $C$ in terms of $q$, $|h| \leq \log ^{\theta} q/2$ and $2/3 \leq r \leq 1$, we have
    \begin{align*}
        \mathbb{E}^{\text{char}}\big[\mathbf{1}_{\mathcal{G}_\chi(h)}|\sum_{\substack{m \leq q^\varepsilon \\ p|m \Rightarrow P_\theta < p \leq P}}\frac{\chi(m)}{m^{1/2+ih}}|^2|\sum_{\substack{n \leq q^{1/2-2\varepsilon} \\ P^{-}(n) > P}}\frac{\chi(n)}{n^{1/2+ih}}|^2\big] \ll\frac{\log ^{1+\theta} q}{1+(1-r)\sqrt{\log \log P}}.
    \end{align*}
    The analogous statement holds for $\mathbb{E}^f$ and Steinhaus $f$.
    \label{good_moment_bound}
\end{proposition}
Equipped with Propositions \ref{chaining}, \ref{bad_chaining_moment} and \ref{good_moment_bound}, we can prove our theorems on the moment of moments. We finally note that it is sufficient to establish our bounds on the range $r \in[2/3,1]$ as then it suffices to apply H\"older's inequality to obtain the result for smaller non negative $r$. The analogous zeta result for Theorem \ref{pf_moments_meso} follows in a near identical fashion, only we use the $t$-aspect version of Lemma \ref{randomisation} instead. We also need Lemma 2, with Zeta Function result 1 from \cite{harper2019partition} and Lemma 3 to deduce suitable versions of Lemmas \ref{factored_form}, \ref{logcor_factored_form} and Proposition \ref{chaining} using the expectation $\mathbb{E}^{\text{cont}}$ (respectively $\mathbb{E}_W^{\text{cont}}$).
\begin{proof}[Proof of Theorems \ref{pf_moments_constant} and \ref{pf_moments_meso}, assuming Propositions \ref{bad_chaining_moment} and \ref{good_moment_bound}]
    This proof follows in three parts: approximate $|L(1/2+ih,\chi)|^2$ as a product of three Dirichlet series, extract the conditioned part and then handle the log-correlated part as in \cite{harper2019partition}. The proof of Theorem \ref{pf_moments_constant} does not need the extra steps in requiring the application of Proposition \ref{chaining}. Note that all applications of Lemma \ref{randomisation} have error $\ll \frac{\log^2 q}{(\log\log P)^{\log\log P}}$ at the very least, which goes to $0$ faster than any power of $\log q$ for $\theta < 0$. An application of Lemma \ref{Dirichlet_AFE} (and dropping the conditioning on $Y_{P_\theta,\chi}(0)$) gives
    \begin{align*}
        &\mathbb{E}_W^{\text{char}}\big(\int_{|h| \leq \log ^{\theta} q/2}|L(1/2+ih,\chi)|^2dh\big)^r 
        \ll \mathbb{E}_W^{\text{char}}\big(\int_{|h| \leq \log ^\theta(q)/2}|\sum_{n \leq \sqrt{q/2\pi}}\frac{\chi(n)}{n^{1/2+ih}}|^2dh\big)^r + \log^{r\theta}q.
    \end{align*}
    Applying the triangle inequality to the first term in the last line provides an upper bound of
    \begin{align*}
        &\mathbb{E}^{\text{char}}\big(\int_{|h| \leq \log ^{\theta} q/2}|\sum_{n \leq \sqrt{q/2\pi}}\frac{\chi(n)}{n^{1/2+ih}}-\sum_{\substack{m \leq q^{\varepsilon_\theta} \\ P^+(m) \leq P_\theta}}\frac{\chi(m)}{m^{1/2+ih}}\sum_{\substack{n \leq q^{1/2-2\varepsilon_\theta} \\ P^-(n) > P_\theta}}\frac{\chi(n)}{n^{1/2+ih}}|^2dh\big)^r \\ 
        +& \mathbb{E}_W^{\text{char}}\big(\int_{|h| \leq \log ^\theta(q)/2}|\sum_{\substack{m \leq q^{\varepsilon_\theta} \\ P^+(m) \leq P_\theta}}\frac{\chi(m)}{m^{1/2+ih}}\sum_{\substack{n \leq q^{1/2-2\varepsilon_\theta} \\ P^-(n) > P_\theta}}\frac{\chi(n)}{n^{1/2+ih}}|^2dh\big)^r
    \end{align*}
    For the first term, by H\"older's inequality and Lemma \ref{factored_form} on the sum of Dirichlet characters it is
    $\ll \big(\log ^{1+2\theta} q/(\log \log  q)^2\big)^r$. We repeat this to access our product of three Dirichlet polynomials. We see 
    \begin{align*}
        &\mathbb{E}_W^{\text{char}}\big(\int_{|h| \leq \log ^\theta(q)/2}|\sum_{\substack{m \leq q^{\varepsilon_\theta} \\ P^+(m) \leq P_\theta}}\frac{\chi(m)}{m^{1/2+ih}}\sum_{\substack{n \leq q^{1/2-2\varepsilon_\theta} \\ P^-(n) > P_\theta}}\frac{\chi(n)}{n^{1/2+ih}}|^2dh\big)^r \\
        \ll& \mathbb{E}_W^{\text{char}}\big(\int_{|h| \leq \frac{\log ^\theta(q)}{2}}|\sum_{\substack{m \leq q^{\varepsilon_\theta} \\ P^+(m) \leq P_\theta}}\frac{\chi(m)}{m^{1/2+ih}}|^2 \\& \quad \qquad\cdot|\sum_{\substack{n \leq q^{1/2-2\varepsilon_\theta} \\ P^{-}(n) >P_\theta}}\frac{\chi(n)}{n^{\sigma+ih}}-\sum_{\substack{k \leq q^{\varepsilon} \\ p|k \Rightarrow P_\theta < p \leq P}}\frac{\chi(k)}{k^{1/2+ih}}\sum_{\substack{n \leq q^{1/2-2\varepsilon} \\ P^{-}(n) > P}}\frac{\chi(n)}{n^{1/2+ih}}|^2dh\big)^r \\
        +& \mathbb{E}_W^{\text{char}}\big(\int_{|h| \leq \log ^\theta(q)/2}|\sum_{\substack{m \leq q^{\varepsilon_\theta} \\ P^+(m) \leq P_\theta}}\frac{\chi(m)}{m^{1/2+ih}}\sum_{\substack{k \leq q^{\varepsilon} \\ p|k \Rightarrow P_\theta < p \leq P}}\frac{\chi(k)}{k^{1/2+ih}}\sum_{\substack{n \leq q^{1/2-2\varepsilon} \\ P^{-}(n) > P}}\frac{\chi(n)}{n^{1/2+ih}}|^2dh\big)^r.
    \end{align*}
    The first term requires us to apply H\"older's inequality so we can take the expectation inside the integral, at which point we use Lemma \ref{partition_result} to approximate the conditioning on the small primes by a smooth partition of unity $t$, and then we transfer over to a statement about random multiplicative functions using Lemma \ref{randomisation}. Since we are now in the random multiplicative function setting, we can expand out the definition of $g_{P_\theta}(Y_{P_\theta,f}(0)) \leq C( \mathbf{1}_{Y_{P_\theta,f}(0) \in [W-1,W+2]}+\delta)$ for $\delta=\log^{\theta} q$ and notice that $\mathbf{1}_{Y_{P_\theta,f}(0) \in [W-1,W+2]}$ only affects the sum over the primes less than $P_\theta$. This allows us to use the independence of the $f(p)$ to note that
    \begin{equation*}\begin{split}
        &\mathbb{E}_W^f|\sum_{\substack{m \leq q^{\varepsilon_\theta} \\ P^+(m) \leq P_\theta}}\frac{f(m)}{m^{1/2+ih}}|^2|\sum_{\substack{n \leq q^{1/2-2\varepsilon_\theta} \\ P^{-}(n) >P_\theta}}\frac{f(n)}{n^{1/2+ih}}-\sum_{\substack{k \leq q^{\varepsilon} \\ p|k \Rightarrow P_\theta < p \leq P}}\frac{f(k)}{k^{1/2+ih}}\sum_{\substack{n \leq q^{1/2-2\varepsilon} \\ P^{-}(n) > P}}\frac{f(n)}{n^{1/2+ih}}|^2 \\
        =& \mathbb{E}_W^f|\sum_{\substack{m \leq q^{\varepsilon_\theta} \\ P^+(m) \leq P_\theta}}\frac{f(m)}{m^{1/2+ih}}|^2\mathbb{E}^f|\sum_{\substack{n \leq q^{1/2-2\varepsilon_\theta} \\ P^{-}(n) >P_\theta}}\frac{f(n)}{n^{1/2+ih}}-\sum_{\substack{k \leq q^{\varepsilon} \\ p|k \Rightarrow P_\theta < p \leq P}}\frac{f(k)}{k^{1/2+ih}}\sum_{\substack{n \leq q^{1/2-2\varepsilon} \\ P^{-}(n) > P}}\frac{f(n)}{n^{1/2+ih}}|^2.
    \end{split}     
    \end{equation*}
    Proposition \ref{chaining} and Lemma \ref{logcor_factored_form} show this is
    $\ll \big(e^{2W}\log ^{1+2\theta} q/(\log \log  q)^2\big)^r.$
    Now for $1-\frac{1}{\sqrt{\log \log  q}} \leq r \leq 1$, we can apply H\"older's inequality to the sum over Dirichlet characters, and we find
    \begin{align*}
        &\mathbb{E}_W^{\text{char}}\big(\int_{|h| \leq \log ^\theta(q)/2}|\sum_{\substack{m \leq q^{\varepsilon_\theta} \\ P^+(m) \leq P_\theta}}\frac{\chi(m)}{m^{1/2+ih}}\sum_{\substack{k \leq q^{\varepsilon} \\ p|k \Rightarrow P_\theta < p \leq P}}\frac{\chi(k)}{k^{1/2+ih}}\sum_{\substack{n \leq q^{1/2-2\varepsilon} \\ P^{-}(n) > P}}\frac{\chi(n)}{n^{1/2+ih}}|^2dh\big)^r \\
        \ll&\big(\int_{|h| \leq \log ^\theta(q)/2}\mathbb{E}_W^{\text{char}}|\sum_{\substack{m \leq q^{\varepsilon_\theta} \\ P^+(m) \leq P_\theta}}\frac{\chi(m)}{m^{1/2+ih}}\sum_{\substack{k \leq q^{\varepsilon} \\ p|k \Rightarrow P_\theta < p \leq P}}\frac{\chi(k)}{k^{1/2+ih}}\sum_{\substack{n \leq q^{1/2-2\varepsilon} \\ P^{-}(n) > P}}\frac{\chi(n)}{n^{1/2+ih}}|^2dh\big)^r \\
        \ll& \big(\int_{|h| \leq \log ^\theta(q)/2}\mathbb{E}^f \big[g_{P_\theta}(h)|\sum_{\substack{m \leq q^{\varepsilon_\theta} \\ P^+(m) \leq P_\theta}}\frac{f(m)}{m^{1/2+ih}}|^2\big]\mathbb{E}^f|\sum_{\substack{k \leq q^{\varepsilon} \\ p|k \Rightarrow P_\theta < p \leq P}}\frac{f(k)}{k^{1/2+ih}}\sum_{\substack{n \leq q^{1/2-2\varepsilon} \\ P^{-}(n) > P}}\frac{f(n)}{n^{1/2+ih}}|^2dh\big)^r \\
        \ll& \big(\log ^{1+\theta} q\int_{|h| \leq \log ^\theta(q)/2}\mathbb{E}^f\big[g_{P_\theta}(h)|\sum_{\substack{m \leq q^{\varepsilon_\theta} \\ P^+(m) \leq P_\theta}}\frac{f(m)}{m^{1/2+ih}}|^2\big]dh\big)^r. 
    \end{align*}
    This is an appropriate form to apply Proposition \ref{chaining}, and the above then is
    $\ll\big(e^{2W}\log ^{1+2\theta} q\big)^r$.\par
    Now, for the remaining values of $r \in [2/3,1-\frac{1}{\sqrt{\log \log  q}}]$, we use an argument first used in \cite{aharper201} which follows by multiple applications of H\"older's inequality. For the remaining $r$, we set
    \begin{align*}
        R(\alpha) :=& \sup_{1-2\alpha \leq r \leq 1-\alpha}\bigg[\big(\frac{(1-r)\sqrt{\log \log  q}}{e^{2W}\log ^{1+2\theta} q}\big)^r \\
         \quad \cdot& \mathbb{E}_W^{\text{char}}\big(\int_{|h|\leq \log ^\theta(q)/2}|\sum_{\substack{m \leq q^{\varepsilon_\theta} \\ P^+(m) \leq P_\theta}}\frac{\chi(m)}{m^{1/2+ih}}\sum_{\substack{k \leq q^{\varepsilon} \\ p|k \Rightarrow P_\theta < p \leq P}}\frac{\chi(k)}{k^{1/2+ih}}\sum_{\substack{n \leq q^{1/2-2\varepsilon} \\ P^{-}(n) > P}}\frac{\chi(n)}{n^{1/2+ih}}|^2dh\big)^r\bigg].
    \end{align*}
    We split this according to whether $\mathcal{G}_\chi$ holds or not. On this event holding, we apply H\"older's inequality, approximate the indicator $\mathbf{1}_{Y_{P_\theta,f}(0) \cap \mathcal{G}_\chi}$ by a smooth function $t$ provided to us by Lemma \ref{partition_result}. We also recall that the smoothing function $g$ is a product of smooth majorants, so we can split it using the decomposition $g(h) = g_{P_\theta}(Y_{P_\theta,f}(0))g_{\log}(h)$ and then can pass to random multiplicative functions by Lemma \ref{randomisation}. Then we can upper bound $g$ from Lemma \ref{partition_result} (using $\delta_{\log}=(\log \log P)^{-1}$ say), and apply independence of prime ranges affected so we can finally apply Proposition \ref{good_moment_bound}. Note that the $((a_i(p),a_i(p^2))_{p \le P_\theta})_{i \le Y}$ are all set to $0$ which allows for the factoring to go through. Putting all these steps together yields
    \begin{align*}
        &\mathbb{E}^{\text{char}}\big(\mathbf{1}_{\mathcal{G}_\chi \cap \{Y_{P_\theta,\chi}(0) \in [W,W+1]\}}\int_{|h|\leq \log ^\theta(q)/2}|\sum_{\substack{m \leq q^{\varepsilon_\theta} \\ P^+(m) \leq P_\theta}}\frac{\chi(m)}{m^{1/2+ih}}\sum_{\substack{k \leq q^{\varepsilon} \\ p|k \Rightarrow P_\theta < p \leq P}}\frac{\chi(k)}{k^{1/2+ih}}\sum_{\substack{n \leq q^{1/2-2\varepsilon} \\ P^{-}(n) > P}}\frac{\chi(n)}{n^{1/2+ih}}|^2dh\big)^r \\
        \ll &\big(\int_{|h|\leq \log ^\theta(q)/2}\mathbb{E}^f\bigg[g_{P_\theta}(Y_{P_\theta,f}(0))|\sum_{\substack{m \leq q^{\varepsilon_\theta} \\ P^+(m) \leq P_\theta}}\frac{f(m)}{m^{1/2+ih}}|^2\bigg] \\
        \qquad& \cdot \mathbb{E}^f\bigg[(\mathbf{1}_{\mathcal{G}(h)}(h)+\delta_{\log})(1+\delta_{\log})^{\log\log(P)}|\sum_{\substack{k \leq q^{\varepsilon} \\ p|k \Rightarrow P_\theta < p \leq P}}\frac{f(k)}{k^{1/2+ih}}\sum_{\substack{n \leq q^{1/2-2\varepsilon} \\ P^{-}(n) > P}}\frac{f(n)}{n^{1/2+ih}}|^2\bigg]dh\big)^r \\
        \ll&\big(\frac{\log^{1+\theta}q}{(1-r)\sqrt{\log\log P}}\int_{|h| \leq \log ^\theta(q)/2}\mathbb{E}^f\bigg[g_{P_\theta}(Y_{P_\theta,f}(0))|\sum_{\substack{m \leq q^{\varepsilon_\theta} \\ P^+(m) \leq P_\theta}}\frac{f(m)}{m^{1/2+ih}}|^2\bigg]dh\big)^r
    \end{align*}
    where we achieved the last line by applying Proposition \ref{good_moment_bound}.
    Once again, we look to Proposition \ref{chaining}, so we obtain this is
        $\ll \Big(\frac{e^{2W}\log^{1+2\theta} q}{(1-r)\sqrt{\log\log q}}\Big)^r.$
    Now, when $\mathcal{G}_\chi$ does not hold, then set $r' = (1+r)/2$, so $1-\alpha \leq r' \leq 1-\alpha/2$. Then we apply H\"older's inequality with exponents $r'/(r'-r)$ and $r'/r$ to get
    \begin{align*}
        &\mathbb{E}_W^{\text{char}}\big(\mathbf{1}_{\mathcal{G}_\chi \text{ fails}}\int_{|h|\leq \log ^\theta(q)/2}|\sum_{\substack{m \leq q^{\varepsilon_\theta} \\ P^+(m) \leq P_\theta}}\frac{\chi(m)}{m^{1/2+ih}}\sum_{\substack{k \leq q^{\varepsilon} \\ p|k \Rightarrow P_\theta < p \leq P}}\frac{\chi(k)}{k^{1/2+ih}}\sum_{\substack{n \leq q^{1/2-2\varepsilon} \\ P^{-}(n) > P}}\frac{\chi(n)}{n^{1/2+ih}}|^2dh\big)^r \\
        \ll& \big(\mathbb{E}^{\text{char}}\mathbf{1}_{\mathcal{G}_\chi \text{ fails}}\big)^{(r'-r)/r'} \\
        \cdot& \big(\mathbb{E}_W^{\text{char}}\big(\int_{|h|\leq \log ^\theta(q)/2}|\sum_{\substack{m \leq q^{\varepsilon_\theta} \\ P^+(m) \leq P_\theta}}\frac{\chi(m)}{m^{1/2+ih}}\sum_{\substack{k \leq q^{\varepsilon} \\ p|k \Rightarrow P_\theta < p \leq P}}\frac{\chi(k)}{k^{1/2+ih}}\sum_{\substack{n \leq q^{1/2-2\varepsilon} \\ P^{-}(n) > P}}\frac{\chi(n)}{n^{1/2+ih}}|^2dh\big)^{r'}\big)^{r/r'}.
    \end{align*}
    The contribution from the first term is $\ll e^{-C/2}$ using that $(r'-r)/r' \geq \alpha/2$ and from Proposition \ref{bad_chaining_moment}, we have $\mathbb{E}^{\text{char}}\big(\mathbf{1}_{\mathcal{G}_\chi \text{ fails}}\big) \ll e^{-C/\alpha}$. In particular, we have
    $$R(\alpha) \ll Ce + e^{-C/2}R(\alpha/2) \ll 1+R(1/\sqrt{\log\log q}) \ll 1$$
    from our previous discussion on the regime $1-\frac{1}{\sqrt{\log \log  q}} \leq r \leq 1$ and recursion. Note that the uses of Lemma \ref{randomisation} have a twist length of
    $$q^{1/2-\varepsilon}\exp\Big(400Y\max\{\log^{3|\theta|}q,(\log\log q)^3\}+1200Y^2\log q(\log\log q)^{-3}\Big)$$
    which is certainly small enough as $Y \asymp \log\log q$ and $N_{\log} = 2.6\log\log q$.
\end{proof}
\subsection{Proof of the Propositions \ref{bad_chaining_moment}  and \ref{good_moment_bound}} 
\begin{proof}[Proof of Proposition \ref{bad_chaining_moment}]
    For convenience, it is much easier to write $\mathcal{G}_\chi$ as a statement of a sum of logarithms of partial Euler products. From the union bound over all the $h(j)$, we can upper bound $\mathbb{P}_{q}(\mathcal{G}_\chi \text{ fails})$ by
    \begin{equation}
        \sum_{j=0}^{\lfloor\log P_q(\theta)\rfloor -B-1}\sum_{h(j)}\mathbb{P}_q\Big(\big|\sum_{\ell=j}^{\lfloor\log P_q(\theta)\rfloor -B-1}\log |I_{\ell,\chi}(h(\ell))|\big|>\log P_q(\theta) -j + t(j)\Big). \label{union_fail}
    \end{equation} 
    The second sum is over all the possible positions of the $h(j)$ over the interval $|h| \leq \log ^\theta q$, where there are $\ll (P_q(\theta)/e^j)(\log(P_q(\theta)/e^j)$ such points, and once we know $h(j)$ then the points $h(\ell)$ are uniquely determined for $\ell \geq j$. Taylor expanding $\log |I_\ell,\chi(h)|)$ reveals the sum over the $\ell$ is equal to
    \begin{align*}
        -\Re&\big(\sum_{\ell=j}^{\lfloor\log P_q(\theta)\rfloor - B -1}\sum_{P^{e^{-(\ell+1)}} < p \leq P^{e^{-\ell}}}\log (1-\frac{\chi(p)}{p^{1/2+ih(\ell)}})\big) \\
        =\Re&\big(\sum_{\ell=j}^{\lfloor\log P_q(\theta)\rfloor - B -1}\sum_{\substack{P^{e^{-(\ell+1)}} < p \leq P^{e^{-\ell}}}}\frac{\chi(p)}{p^{1/2+ih(\ell)}}+\frac{(\chi(p))^2}{2p^{1+2ih(\ell)}}+O(p^{-3/2})\big).
    \end{align*}
    Consequently, if $\mathcal{G}_\chi$ fails, then we must have
    $$\big|\sum_{\ell=j}^{\lfloor\log P_q(\theta)\rfloor - B -1}\sum_{P^{e^{-(\ell+1)}} < p \leq P^{e^{-\ell}}}\big(\frac{\chi(p)}{p^{1/2+ih(\ell)}}+\frac{(\chi(p))^2}{2p^{1+2ih(\ell)}}\big)\big| > \log P_q(\theta) + t(j)-j+O(1).$$
    An application of Markov's inequality to the $2k$-th power where $$k=\big\lfloor\frac{(\log P_q(\theta)+t(j)-j+O(1))^2}{\log P_q(\theta)-j+O(1)}\big\rfloor$$
    combined with Stirling approximation yields
    \begin{align*}
        \sqrt{k}\exp\big(-\big\lfloor\frac{(\log P_q(\theta)+t(j)-j+O(1))^2}{\log P_q(\theta)-j+O(1)}\big\rfloor\big) \ll \sqrt{k}\frac{e^j}{P_q(\theta)}e^{-2t(j)-\frac{(t(j)+O(1))^2}{\log P_q(\theta)-j+O(1)}} \numberthis \label{excep_set}
    \end{align*}
    where this has been achieved by an application of Lemma \ref{even_moment}. We deduce this value of $k$ by optimising the choice of $k$ for the $2k$-th moment bound and balancing this with the Stirling approximation. Now, we split into the cases dependent on the size of $2C\min\{\sqrt{\log P_q(\theta)},(1-r)^{-1}\}$. Now if $C\min\{\sqrt{\log P_q(\theta)},(1-r)^{-1}\} < \log P_q(\theta)-j$, then $k \asymp \log P_q(\theta)-j$ so we obtain the bound $$\ll \frac{e^j}{P_q(\theta)}\frac{1}{\log ^{7/2}(P_q(\theta)/e^j)}e^{-2C\min\{\sqrt{\log P_q(\theta)},(1-r)^{-1}\}}.$$ Otherwise, $C\min\{\sqrt{\log P_q(\theta)},(1-r)^{-1}\} \geq \log P_q(\theta)-j$, then we have that one can bound equation (\ref{excep_set}) by $$\frac{e^j}{P_q(\theta)}\frac{1}{\log ^{5}(P_q(\theta)/e^j)}e^{-3C\min\{\sqrt{\log P_q(\theta)},(1-r)^{-1}\}}.$$
    Regardless, we see that equation (\ref{excep_set}) is $$\ll \frac{e^j}{P_q(\theta)}\frac{1}{\log ^{7/2}(P_q(\theta)/e^j)}e^{-2C\min\{\sqrt{\log P_q(\theta)},(1-r)^{-1}\}}.$$ Putting this back into equation (\ref{union_fail}), we see
    \begin{align*}
        \mathbb{P}_q(\mathcal{G}_\chi \text{ fails}) \ll& \sum_{j=0}^{\lfloor\log P_q(\theta)\rfloor -B-1}\sum_{h(j)}\frac{e^j}{P_q(\theta)}\frac{1}{\log ^{7/2}(P_q(\theta)/e^j)}e^{-2C\min\{\sqrt{\log P_q(\theta)},(1-r)^{-1}\}} \\
        \ll& \sum_{j=0}^{\lfloor\log P_q(\theta)\rfloor -B-1}\frac{1}{\log ^{5/2}(P_q(\theta)/e^j)}e^{-2C\min\{\sqrt{\log P_q(\theta)},(1-r)^{-1}\}}
    \end{align*}
    as the sum converges.
\end{proof}
\begin{proof}[Proof of Proposition \ref{good_moment_bound}]
To deduce the parameters we will need to apply Lemma \ref{partition_result}, we recall that
$$\log \big(\prod_{\ell=j}^{\lfloor\log P_q(\theta)\rfloor - B -1}|I_{\ell,\chi}(h(\ell))|\big) = \Re(\big(\sum_{\ell=j}^{\lfloor\log P_q(\theta)\rfloor - B -1}\sum_{\substack{P^{e^{-(\ell+1)}} < p \leq P^{e^{-\ell}}}}\sum_{v \geq 1}\frac{(\chi(p))^v}{vp^{v(1/2+ih(\ell))}}\big)).$$
The terms corresponding to $v \geq 3$ are convergent and can be bounded by $8/3$. This then gives us our choice of $N_j$, which is
$$N_j = \log P_q(\theta)-j+t(j)+8/3.$$
We will choose $\delta>0$ later. Then we can bound $\mathbf{1}_{\mathcal{G}_\chi(h)}(h)$ by
$$\prod_{j=0}^{\lfloor\log P_q(\theta)\rfloor-B-1}G_{N_j}\big(\Re\big(\sum_{\ell=j}^{\lfloor\log P_q(\theta)\rfloor - B -1}\sum_{\substack{P^{e^{-(\ell+1)}} < p \leq P^{e^{-\ell}} \\ v=1,2}}\frac{(\chi(p))^v}{vp^{v(1/2+ih(\ell))}}\big)\big).$$
Here we denote $G_{N_j}(x) = \frac{1}{1-\delta_j}\sum_{|i| \le N_j}g_{j}^{(i)}(x)$ as the partition of unity for the index $j$ on the interval $[-N_j,N_j+1]$. Then assuming that we can apply Lemma \ref{randomisation}, we have
\begin{align*}
    &\mathbb{E}^{\text{char}}\Big[\prod_{j=0}^{\lfloor\log P_q(\theta)\rfloor-B-1}G_{N_j}(S_{j,\theta}(\chi))\big|\sum_{\substack{k \le q^{\varepsilon} \\ p|k \implies P_\theta < p \le P}}\frac{\chi(k)}{k^{1/2+ih}}\big|^2\big|\sum_{\substack{n \le q^{1/2-2\varepsilon} \\ P^-(n) > P}}\frac{\chi(n)}{n^{1/2+ih}}\big|^2\Big] \\
    =& \mathbb{E}^{f}\Big[\prod_{j=0}^{\lfloor\log P_q(\theta)\rfloor-B-1}G_{N_j}(S_{j,\theta}(f))|\sum_{\substack{k \leq q^\varepsilon \\ p|k \Rightarrow P_\theta < p \leq P}}\frac{f(k)}{k^{1/2+ih}}|^2|\sum_{\substack{n \leq q^{1/2-2\varepsilon} \\ P^{-}(n) > P}}\frac{f(n)}{n^{1/2+ih}}|^2\Big]\\+&O(\frac{N_j\log P_\theta}{(\log P_q(\theta)\log P)^{\log P_q(\theta))/\delta^2}}),
\end{align*}
where we have set
$$S_{j,\theta}(\chi) = \Re\big(\sum_{\ell=j}^{\lfloor\log P_q(\theta)\rfloor - B -1}\sum_{\substack{P^{e^{-(\ell+1)}} < p \leq P^{e^{-\ell}} \\ v=1,2}}\frac{(\chi(p))^v}{vp^{v(1/2+ih(\ell))}}\big)$$
for compactness. This error term is more than acceptable for our theorem with our eventual choice of $\delta$ as $1 \ll N_j \ll \log\log P$ that is sufficient with the choice of $\delta$. Then we can use Property $2.$ of the $g$, so we can reverse the smoothing, and we find the above is
\begin{align*}
    &\mathbb{E}^{\text{char}}\big[\mathbf{1}_{\mathcal{G}_\chi}|\sum_{\substack{k \leq q^\varepsilon \\ p|k \Rightarrow P_\theta < p \leq P}}\frac{\chi(k))}{k^{1/2+ih}}|^2|\sum_{\substack{n \leq q^{1/2-2\varepsilon} \\ P^{-}(n) > P}}\frac{\chi(n)}{n^{1/2+ih}}|^2\big] \numberthis \label{randomed}\\ \ll& \mathbb{E}^f\big[(1+\delta)^{\log P_q(\theta)}(\mathbf{1}_{\mathcal{G}_{f}(h)}(h)+\delta)|\sum_{\substack{k \leq q^\varepsilon \\ p|k \Rightarrow P_\theta < p \leq P}}\frac{f(k)}{k^{1/2+ih}}|^2|\sum_{\substack{n \leq q^{1/2-2\varepsilon} \\ P^{-}(n) > P}}\frac{f(n)}{n^{1/2+ih}}|^2\big].
\end{align*}
Choosing $\delta = (\log \log P)^{-1}$ for example then shows that
$$(1+\delta)^{\log P_q(\theta)}(\mathbf{1}_{\mathcal{G}_{f}(h)}(h)+\delta) \ll \mathbf{1}_{\mathcal{G}_{f}(h)}(h)+ \delta.$$
Then using some elementary random Euler product estimates, equation (\ref{randomed}) is
\begin{align*}
    \ll& \mathbb{E}^f\big[(\mathbf{1}_{\mathcal{G}_{f}(h)}(h)+\delta)|\sum_{\substack{k \leq q^\varepsilon \\ p|k \Rightarrow P_\theta < p \leq P}}\frac{f(k)}{k^{1/2+ih}}-\prod_{P_\theta < p \le P}(1-\frac{f(p)}{p^{1/2+ih}})^{-1}|^2|\sum_{\substack{n \leq q^{1/2-2\varepsilon} \\ P^{-}(n) > P}}\frac{f(n)}{n^{1/2+ih}}|^2\big] \\
    \qquad &+ \mathbb{E}^f\big[(\mathbf{1}_{\mathcal{G}_{f}(h)}(h)+\delta)\prod_{P_\theta < p \le P}|1-\frac{f(p)}{p^{1/2+ih}}|^{-2}|\sum_{\substack{n \leq q^{1/2-2\varepsilon} \\ P^{-}(n) > P}}\frac{f(n)}{n^{1/2+ih}}|^2\big] \\
    \ll& \frac{\log^{1+\theta} q}{\log\log q} + \mathbb{E}^f\big[\mathbf{1}_{\mathcal{G}_{f}(h)}(h)\prod_{P_\theta < p \le P}|1-\frac{f(p)}{p^{1/2+ih}}|^{-2}\big]\mathbb{E}^f|\sum_{\substack{n \leq q^{1/2-2\varepsilon} \\ P^{-}(n) > P}}\frac{f(n)}{n^{1/2+ih}}|^2 \ll\frac{\log^{1+\theta} q}{1+(1-r)\sqrt{\log\log q}}
\end{align*}
where we have used orthogonality of the $f(p)$, a smooth number estimate (similar to those used in Lemma \ref{factored_form}) and applied an appropriate form of Key Proposition 1 from \cite{aharper201} with the number of scales considered being $$n_\theta = \log\log P +\theta \log\log q +O(1)= (1+\theta)\log\log q - 8\log\log\log q+O(1)$$ which is suitable for our purposes.
It remains to check whether our choice of parameters were suitable for applying Lemma \ref{randomisation}. Indeed, we have that $$xP^{400(Y/\delta)^2\log(N\log P)} \ll q^{1/2-\varepsilon+\varepsilon_\theta}P^{1200(\log \log P)^5} \ll q^{1/2-\varepsilon+\varepsilon_\theta}q^{1/\log \log  q} = q^{1/2+o(1)}$$ which is again sufficient for Lemma \ref{randomisation}.
\end{proof}
\subsection{Proof of the Corollaries}
\begin{proof}[Proof of Corollaries \ref{L_constant}, \ref{typical_partition_meso} and \ref{typical_partition_zeta}]
    Partitioning the event on the size of $Y_{P_\theta,\chi}(0)$, then
    \begin{align*}
        &\#\Big\{\chi \bmod{q} : \mathbf{1}_{Y_{P_\theta,\chi}(0) \le W}\mathcal{Z}_2(q,\chi,\theta) > \frac{\lambda e^{2W}\log^{1+2\theta}q}{\sqrt{\log\log q}}\Big\} \\
        =& \#\Big\{\chi \bmod{q} : \mathbf{1}_{W_0 < Y_{P_\theta,\chi}(0) \le W}\mathcal{Z}_2(q,\chi,\theta) > \frac{\lambda e^{2W}\log^{1+2\theta}q}{\sqrt{\log\log q}}\Big\}\\ & \qquad+ \#\Big\{\chi \bmod{q} : \mathbf{1}_{-C\log\log q< Y_{P_\theta,\chi}(0) \le W_0}\mathcal{Z}_2(q,\chi,\theta) > \frac{\lambda e^{2W}\log^{1+2\theta}q}{\sqrt{\log\log q}}\Big\} \\
        & \qquad +\#\Big\{\chi \bmod{q} : \mathbf{1}_{Y_{P_\theta,\chi}(0) \le -C\log\log q }\mathcal{Z}_2(q,\chi,\theta) > \frac{\lambda e^{2W}\log^{1+2\theta}q}{\sqrt{\log\log q}}\Big\},
    \end{align*}
    where $C$ is some large constant chosen later.
    For the first two events, we can apply Markov's inequality to the power $r=1-\frac{1}{2\log  \lambda}$. In the case of the second quantity, we look at a far less localised version for $Y_{P_\theta,\chi}(0)$, which leads to the twist length being
    $$q^{1/2}\exp\Big(400C\log^{3|\theta|}q(\log\log q)^4+\frac{1200\log q}{\log\log q}\Big),$$
    which is more than sufficient for application of Lemma \ref{randomisation}.\par
    To handle the final term, we take the $r=3/4$ moment, and apply Markov to get
    \begin{align*}
        &\frac{1}{q-1}\#\Big\{\chi \bmod{q} : \mathbf{1}_{Y_{P_\theta,\chi}(0) \le -C\log\log q }\mathcal{Z}_2(q,\chi,\theta) > \frac{\lambda e^{2W}\log^{1+2\theta}q}{\sqrt{\log\log q}}\Big\} \\\ll& \frac{(\log\log q)^{3/8}\mathbb{E}^{\text{char}}[(\mathbf{1}_{Y_{P_\theta,\chi}(0) \le -C\log\log q }\mathcal{Z}_2(q,\chi,\theta))^{3/4}]}{\lambda^{3/4}e^{3W/2}\log^{3(1+2\theta)/4}q} \\
        \ll& \frac{(\log\log q)^{3/8}(\mathbb{E}^{\text{char}}[\mathbf{1}_{Y_{P_\theta,\chi}(0) \le -C\log\log q}])^{1/4}(\mathbb{E}^{\text{char}}|\mathcal{Z}_2(q,\chi,\theta)|)^{3/4}}{e^{3W/2}\lambda^{3/4}\log^{3(1+2\theta)/4} q}.
    \end{align*}
    where we have applied H\"older's inequality on the second line. The expectation involving $\mathcal{Z}_2(q,\chi,\theta)$ is $\ll \log^{3(1+\theta)/4}(q)$. As for the expectation involving $Y_{P_\theta,\chi}(0)$ being small, we look to its high Gaussian moments. We have
    $$\mathbb{E}^{\text{char}}[\mathbf{1}_{Y_{P_\theta,\chi}(0) < -C\log\log q} ] \le \mathbb{P}_q(|Y_{P_\theta,\chi}(0)| > C\log\log q) \ll \frac{(k!)(|\theta|\log\log q + O(1))^{k}}{(C\log\log q)^{2k}}$$
    for $k = \lfloor (C\log\log q +O(1))^2/(|\theta|\log\log q+O(1))\rfloor$ (which we have $P_\theta^{2k} < q$ clearly, so Lemma \ref{even_moment} holds). For $C=100$, then this is from the Stirling approximation
    $$\ll \frac{(\log\log q)^{3/8}\log^{3|\theta|/4-24}q}{e^{3W/2}\lambda^{3/4}}$$
    which is certainly sufficient for the range of $\lambda$ we are considering. \par
    For $\lambda \ge e^{\sqrt{\log\log q}}$, then one can instead use a first moment estimate in Corollary \ref{L_constant}.
\end{proof}
\section{Proof of Theorem \ref{short_range_max}} \label{typical_max_section}
To state our assumed propositions, we need to define a suitable barrier event on the log-correlated part of our Dirichlet polynomial. This will be very similar to the event $\mathcal{G}_\chi$ defined in Section \ref{pfm_proof}. Again, fix $P=q^{1/(\log\log q)^8}$ and $\varepsilon=1/(\log\log q)^{2}$. We introduce the following set of discrete points which we use to approximate $h$ at various levels of precision which is slightly tighter than the ones we used for the partition function bounds. Let $\tilde{h}(-1) =h$ and for $0 \leq j \leq \log \log P-1$, define
$$\tilde{h}(j) := \max\big\{u \leq \tilde{h}(j-1):u=\frac{n}{(\log P/e^j)(\log(\log P/e^j))^3} \text{ for some } n \in \mathbb{Z}\big\}.$$
We also define $\tilde{h}(*)$ as the point $\frac{n}{\log  q}$ which is closest to $h$. We set $I_{\ell,\chi}(h)$ and $I_{\ell,f}(h)$ as they are defined in equations (\ref{char_increment}). Then, we let $\widetilde{\mathcal{G}}_\chi$ denote the event that for all $0 \leq j \leq \log \log P-4\log(2+U)-B-1=T(P,U)$
$$\Big|\sum_{\ell =j}^{T(P,U)}\log I_{\ell,\chi}(\tilde{h}(\ell)))\Big| \leq \log \log P-j+6\log (\log \log P-j )+U +B$$
for all $|h| \leq 1$ with $B$ some large absolute constant and $U\geq 1$ a parameter. We let $\widetilde{\mathcal{G}}_\chi(h)$ denote the event that this occurs for a particular $h$ (rather than all of them). We further introduce the loose event
$$\mathcal{A}_\chi := \Big\{\sup_{|h| \le 1}\Big|\sum_{p \le P}\frac{\chi(p)}{p^{1/2+ih}}\Big| \le 2C\log\log P+CU\Big\}.$$
We use this so we can compare Dirichlet polynomials supported on smooth numbers being large to its corresponding Euler product being large, since we can only transport conditioning on Euler products across using Lemma \ref{randomisation}, and $C$ corresponds to an absolute constant supplied in Lemma 26 of \cite{abr20}.\par
The change to look at $I_{\ell,\chi}(h)$ rather than $|I_{\ell,\chi}(h)|$ will simplify the proof of bounding the number of $\chi$ such that $\widetilde{\mathcal{G}}_\chi$ fails. 
\begin{proposition}
    Uniformly for $1 \leq U \leq 2\log \log P$,
    $$\mathbb{P}_q\big(((\mathcal{A}_\chi \cap\widetilde{\mathcal{G}}_\chi) \ \text{fails}\big) \ll \sqrt{U}e^{-2U}.$$ \label{max_event_fail}
\end{proposition}
\begin{proposition}
    Uniformly for all $\ 1 \leq U \leq 2\log \log P$ and $V \geq e^{-U}$; and for all $|h| \leq 1$ and $|\sigma-1/2| \leq 1/\log  q$, we have
    \begin{align*}
        &\mathbb{E}^{\text{char}}\big[\mathbf{1}_{\mathcal{A}_\chi \cap \widetilde{\mathcal{G}}_\chi(h)}(h)\mathbf{1}_{\mathcal{H}_\chi(h;V)}(h)\big|\sum_{\substack{m \leq q^\varepsilon \\ P^+(m) \leq P}}\frac{\chi(m)}{m^{1/2+ih}}\sum_{\substack{n \leq q^{1/2-2\varepsilon} \\ P^-(n) > P}}\frac{\chi(n)}{n^{\sigma+ih}}\big|^2\big] \\
        \ll& \log  q\min\big\{1,\frac{U}{\sqrt{\log \log P}}\big\}\big(\min\big\{1,\frac{U+\log \log \log P+\log V}{\sqrt{\log \log P}}\big\}\big)^2. 
    \end{align*} \label{conditional_bound_max}
\end{proposition}
\begin{proof}[Proof of Theorem \ref{short_range_max} assuming Propositions \ref{max_event_fail} and \ref{conditional_bound_max}]
    We set the parameter $V=e^{-U}(\log \log  q)^6$, which is compatible with Proposition \ref{conditional_bound_max}. \par
    For each $\chi$ modulo $q$, we define $h_\chi$ as the value of $h$ such that $\max_{|h| \leq1/2}|L(1/2+ih,\chi)|$ is obtained. From the statement of Theorem \ref{short_range_max}, we will need to prove
    \begin{equation}
        \mathbb{P}_q\big(|L(1/2+ih_\chi,\chi)|>\frac{e^{U}\log  q}{(\log \log  q)^{3/4}}\big) \ll Ue^{-2U}(\log \log \log  q)^2. \label{target_max}
    \end{equation}
    Now, if $|L(1/2+ih_\chi,\chi)| \geq \frac{e^{U}\log  q}{(\log \log  q)^{3/4}}$, we have by Lemma \ref{Dirichlet_AFE}
    $$\sum_{n \leq \sqrt{q/2\pi}}\frac{\chi(n)}{n^{1/2+ih_\chi}} \gg \frac{e^{U}\log  q}{(\log \log  q)^{3/4}}$$
    as one has by triangle inequality followed by Markov
    $$\mathbb{P}_q\big(|L(1/2+ih_\chi,\chi)|>\frac{e^{U}\log  q}{(\log \log  q)^{3/4}}\big) \ll \frac{e^{-2U}(\log\log q)^{3/2}}{\log q} + \mathbb{P}_q(|\sum_{n \le \sqrt{q/2\pi}}\frac{\chi(n)}{n^{1/2+ih_\chi}}| > \frac{e^U\log q}{3(\log\log q)^{3/4}}).$$
    This Dirichlet polynomial is a holomorphic function, so we can apply the Cauchy integral formula to rewrite this. We then see
    $$\sum_{n \leq \sqrt{q/2\pi}}\frac{\chi(n)}{n^{1/2+ih_\chi}} = \frac{1}{2\pi i}\big(\int_{A_\chi}+\int_{B_\chi}+\int_{C_\chi}+\int_{D_\chi}\big)\sum_{n \leq \sqrt{q/2\pi}}\frac{\chi(n)}{n^{s}}\frac{ds}{(s-1/2-ih_\chi)},$$
    where the union of the lines $A_\chi, B_\chi, C_\chi$ and $D_\chi$ forms a square centred at $1/2+i\tilde{h}_\chi(*)$ and edges of length $\asymp \frac{1}{\log q}$. Consequently for this square with sides $A_\chi, B_\chi, C_\chi, D_\chi$, we have
    $$\max_{\mathcal{C} \in \{A_\chi,B_\chi,C_\chi,D_\chi\}}\int_{\mathcal{C}}\big|\sum_{n \leq \sqrt{q/2\pi}}\frac{\chi(n)}{n^{s}}\big|d|s|  \gg \frac{e^{U}}{(\log \log  q)^{3/4}}.$$
    We investigate the vertical lines first. These can also be handled in an identical fashion, so here $\sigma$ can denote $1/2+1/\log  q$ or $1/2-1/\log  q$. After several applications of the triangle inequality, we see after a change of variables $s \mapsto s+1/2$
    \begin{align*}
        &\int_{|\tilde{h_\chi}(*)-h| \leq 1/\log  q}\big|\sum_{n \leq \sqrt{q/2\pi}}\frac{\chi(n)}{n^{\sigma+ih}}\big|dh  \numberthis \label{core_triangle}\\ 
        \ll& \int_{|\tilde{h_\chi}(*)-h| \leq 1/\log  q}\big|\sum_{n \leq \sqrt{q/2\pi}}\frac{\chi(n)}{n^{\sigma+ih}}-\sum_{\substack{m \leq q^{\varepsilon} \\ P^+(m) \leq P}}\frac{\chi(m)}{m^{1/2+ih}}\sum_{\substack{n \leq q^{1/2-2\varepsilon} \\ P^-(n) > P}}\frac{\chi(n)}{n^{\sigma+ih}}\big|dh \\
        +& \int_{|\tilde{h_\chi}(*)-h| \leq 1/\log  q}|\sum_{\substack{m \leq q^{\varepsilon} \\  P^+(m) \leq P}}\frac{\chi(m)}{m^{1/2+ih}}\sum_{\substack{n \leq q^{1/2-2\varepsilon} \\ P^{-}(n) > P}}\frac{\chi(n)}{n^{\sigma+ih}}|dh.
    \end{align*}
    For the first expression, if it is $\gg \frac{e^{U}}{(\log \log  q)^{3/4}}$, then by applying Cauchy-Schwarz, we must have
    $$\frac{e^{2U}\log q}{(\log \log  q)^{3/2}} \ll \int_{|\tilde{h_\chi}(*)-h| \leq 1/\log  q}\big|\sum_{n \leq \sqrt{q/2\pi}}\frac{\chi(n)}{n^{\sigma+ih}}-\sum_{\substack{m \leq q^{\varepsilon} \\ P^+(m) \leq P}}\frac{\chi(m)}{m^{1/2+ih}}\sum_{\substack{n \leq q^{1/2-2\varepsilon} \\ P^-(n) > P}}\frac{\chi(n)}{n^{\sigma+ih}}\big|^2dh.$$
    We extend the range of integration to $|h| \leq 1$ (so we now have no character dependency), and sum over all Dirichlet characters modulo $q$. The number of $\chi$ where this occurs is then
    $\ll \frac{qe^{-2U}}{\sqrt{\log \log  q}},$
    which is certainly good enough. \par
    The final term in equation (\ref{core_triangle}) requires a more delicate procedure. First, we restrict to the event $\mathcal{A}_\chi \cap \widetilde{\mathcal{G}}_\chi$, as by Markov's inequality followed by Cauchy-Schwarz
    \begin{align*}
        &\#\Big\{\chi \bmod{q} : \int_{|\tilde{h_\chi}(*)-h| \leq 1/\log  q}\mathbf{1}_{\mathcal{A}_\chi \cap \widetilde{\mathcal{G}}_\chi \text{ fails}}|\sum_{\substack{m \leq q^{\varepsilon} \\  P^+(m) \leq P}}\frac{\chi(k)}{k^{1/2+ih}}\sum_{\substack{n \leq q^{1/2-2\varepsilon} \\ P^{-}(n) > P}}\frac{\chi(n)}{n^{\sigma+ih}}|dh > \frac{e^U\log q}{(\log\log q)^{3/4}}\Big\} \\
        \ll& \frac{qe^{-U}(\log\log q)^{3/4}}{\log^{1/2}q}(\mathbb{P}_q((\mathcal{A}_\chi \cap \widetilde{\mathcal{G}}_\chi) \text{ fails}))^{1/2} (\int_{|h| \le 1} \mathbb{E}^{\text{char}}|\sum_{\substack{m \leq q^{\varepsilon} \\  P^+(m) \leq P}}\frac{\chi(k)}{k^{1/2+ih}}\sum_{\substack{n \leq q^{1/2-2\varepsilon} \\ P^{-}(n) > P}}\frac{\chi(n)}{n^{\sigma+ih}}|^2dh)^{1/2} \\ \ll& qU^{1/4}e^{-2U}(\log\log q)^{3/4}
    \end{align*}
    which is sufficiently small for our target bound by Proposition \ref{max_event_fail}. By H\"older's inequality, if the final integral is $\gg \frac{e^{U}}{(\log \log  q)^{3/4}}$, then using the exponents $r=4/3$ and $r'=4$, then we have either
    \begin{align*}
        \frac{e^{4U}\log ^3(q)}{(\log \log  q)^3} \ll \int_{|\tilde{h_\chi}(*)-h| \leq 1/\log  q}\mathbf{1}_{\mathcal{H}_\chi(h;V) \ \text{fails}}(h)\big|\sum_{\substack{m \leq q^\varepsilon \\ P^+(m) \leq P}}\frac{\chi(m)}{m^{1/2+ih}}\sum_{\substack{n \leq q^{1/2-2\varepsilon} \\ P^{-}(n) > P}}\frac{\chi(n)}{n^{\sigma+ih}}\big|^4dh \numberthis \label{small_dirichlet}
    \end{align*}
    or
    \begin{align*}
        &\frac{e^{2U}\log q}{(\log \log  q)^{3/2}} \ll \int_{|\tilde{h_\chi}(*)-h| \leq 1/\log  q}\mathbf{1}_{\mathcal{A}_\chi \cap \widetilde{\mathcal{G}}_\chi}\mathbf{1}_{\mathcal{H}_\chi(h;V)}(h)\big|\sum_{\substack{m \leq q^\varepsilon \\ P^+(m) \leq P}}\frac{\chi(m)}{m^{1/2+ih}}\sum_{\substack{n \leq q^{1/2-2\varepsilon} \\ P^{-}(n) > P}}\frac{\chi(n)}{n^{\sigma+ih}}\big|^2dh. \numberthis \label{big_dirichlet}
    \end{align*}
    We begin with equation (\ref{small_dirichlet}). As before, we extend the range of integration in $h$ to make the quantity independent of $\chi$, and sum over the Dirichlet characters modulo $q$. The integrand in equation (\ref{small_dirichlet}) is then handled by Lemma \ref{conditional_fourth}. From this, we deduce that the number of characters modulo $q$ that satisfy equation (\ref{small_dirichlet}) is 
    $$\ll \frac{qe^{-4U}(\log \log  q)^3\log q}{V^2\log P} \ll \frac{qe^{-2U}}{\log \log  q}.$$
    Finally, we handle equation (\ref{big_dirichlet}). Extending the range of integration over $h$, we have that the number of $\chi$ such that $\mathcal{H}(h;V)$ holds is 
    $$\ll \frac{q(\log \log  q)^{3/2}e^{-2U}}{\log  q}\int_{|h| \leq 1}\mathbb{E}^{\text{char}}\big[\mathbf{1}_{\mathcal{A}_\chi \cap \widetilde{\mathcal{G}}_\chi}(h)\mathbf{1}_{\mathcal{H}_\chi(h;V)}(h)\big|\sum_{\substack{m \leq q^\varepsilon \\ P^+(m) \leq P}}\frac{\chi(m)}{m^{1/2+ih}}\sum_{\substack{n \leq q^{1/2-2\varepsilon} \\ P^{-}(n) > P}}\frac{\chi(n)}{n^{\sigma+ih}}\big|^2\big]dh.$$
    From this, we use $\mathbf{1}_{\mathcal{A}_\chi \cap \widetilde{\mathcal{G}}_\chi}(h) \leq \mathbf{1}_{\mathcal{A}_\chi \cap \widetilde{ \mathcal{G}}_\chi(h)}(h)$, so the above is
    \begin{align*}
        \ll& q(\log \log  q)^{3/2}e^{-2U}\min\big\{1,\frac{U}{\sqrt{\log \log P}}\big\}\big(\min\big\{1,\frac{U+\log \log \log P+\log V}{\sqrt{\log \log P}}\big\}\big)^2 \\
        \ll& qUe^{-2U}(\log \log \log  q)^2.
    \end{align*}
    Now, it only remains to handle the contribution of the horizontal sides. These will be of the form
    $$\int_{|\sigma-1/2| \leq1/\log  q}\big|\sum_{n \leq \sqrt{\frac{q}{2\pi}}}\frac{\chi(n)}{n^{\sigma+i(\tilde{h_\chi}(*)+a/\log  q)}}\big|d\sigma$$
    where $a\in\{-1,+1\}$. Now, we can use the triangle inequality to obtain an upper bound for the above in the same style as equation \ref{core_triangle}. What is quite helpful is that the shift in $\sigma$ only affects the Dirichlet polynomial supported on rough numbers. What does need more work is how we make each expression independent of $\chi$. This is specifically why we introduced the new approximating points $\tilde{h_\chi}(*)$ as these are all of the form $n/\log  q$. We can then sum over all of the possible positions of $\tilde{h_\chi}(*)$ over the interval $[-1,1]$, of which there are $\asymp \log  q$ of these points. One then can proceed in an identical manner to the way we do for the vertical integrals, and obtain the same bounds. This is why all of our preliminary lemmas and propositions only involved averaging over the characters and uniformly in $h$ so they can be applied in this situation.
\end{proof}
\subsection{Proof of Proposition \ref{max_event_fail} and \ref{conditional_bound_max}}
\begin{proof}[Proof of Proposition \ref{max_event_fail}]
    By the triangle inequality, we have
    $$\mathbb{P}_q((\mathcal{A}_\chi \cap \widetilde{\mathcal{G}}_\chi)\ \text{fails}) \le \mathbb{P}_q(\mathcal{A}_\chi \ \text{fails}) + \mathbb{P}_q(\widetilde{\mathcal{G}}_\chi \ \text{fails}),$$
    and we analyse each event separately.
    Most of the proof for the event $\mathbb{P}_q(\tilde{\mathcal{G}}_\chi \text{ fails})$ has already been completed in Proposition \ref{bad_chaining_moment}. Notice that this time, we are interested in bounding absolute values of the sum of $\log I_{\ell,\chi}(\tilde{h}(\ell))$, and not the sums of the real part of $\log I_{\ell,\chi}(\tilde{h}(\ell))$. This is not a problem as we upper bounded the real part by their absolute values in the proof of Proposition \ref{bad_chaining_moment}.
    We want to find the probability such that
    \begin{equation}
        \big|\sum_{\ell=j}^{T(P,U)}\sum_{\substack{P^{e^{-(\ell+1)}} < p \leq P^{e^{-\ell}}}}\frac{\chi(p)}{p^{1/2+i\tilde{h}(\ell)}}+\frac{(\chi(p))^2}{2p^{1+2i\tilde{h}(\ell)}}\big| > \log \log P -j + 6\log(\log\log P -j)+U + B. \label{max_excep_set}
    \end{equation}
    An application of Markov's inequality to the $2r$-th power with
    $$r=\big\lfloor\frac{(\log \log P-j+6\log(\log\log P-j)+U+O(1))^2}{\log \log P-j+O(1)}\big\rfloor$$
    shows that the probability that \ref{max_excep_set} holds is
    $$\ll \sqrt{r}\big(\frac{r(\log \log P-j+O(1))}{e(\log \log P-j+6\log( \log \log P-j)+U+O(1))^2}\big)^r.$$
    As $U \leq 2\log \log P$, then we certainly know that $r \leq (3\log \log P)^2$, so this is bounded by
    $$\ll \sqrt{\log\log P-j+U}\frac{e^j}{\log P}e^{-2U-6\log(\log\log P-j)} \ll \sqrt{U}e^{-2U}\frac{e^j}{\log P(\log\log P-j)^{11/2}}.$$
    Substituting back into the initial union bound we take over the $\tilde{h}(j)$ (recall there are $\ll e^{-j}\log P(\log(\log P/e^{j}))^3$ such points) and over $j$, the final bound is $\ll \sqrt{U}e^{-2U}$. \par
    As for $\mathbb{P}_q(\mathcal{A}_\chi \ \text{fails})$, we proceed by applying Lemma 26 in \cite{abr20} to discretise the supremum in that event, noting that it is a statement about Dirichlet polynomials, so it equally applies in our situation. Consequently, there exists a constant $C$ such that
    $$\mathbb{P}_q(\mathcal{A}_\chi \ \text{fails}) \le \mathbb{P}_q(\max_{h \in \mathcal{D}_P}\Big|\sum_{p \le P}\frac{\chi(p)}{p^{1/2+ih}}\Big| > \frac{C(2\log\log P +U)}{C}) + e^{-100\log\log P}$$
    choosing $A=100$ with $\mathcal{D}_P = [-1,1] \cap e^{-100-\log\log P}\mathbb{Z}$. In particular, there are $\ll \log P$ points in the discretisation set. Once again applying Lemma \ref{even_moment} with
    $$r=\lfloor\frac{(2\log\log P + U + O(1))^2}{\log\log P +O(1)}\rfloor,$$
    we get that
    $$\mathbb{P}_q(\mathcal{A}_\chi \ \text{fails}) \ll \log P \sqrt{U+\log\log P}e^{-2\log\log P -2U} \ll e^{-2U}$$
    using that $U \le 2\log\log P$.
\end{proof}
We need some preparation in order to prove Proposition \ref{conditional_bound_max}, so we can translate large values of a Dirichlet polynomial supported on smooth numbers to large values of Euler products, so one can allow Lemma \ref{randomisation} to transfer to the random side (as it can only take conditioning on Euler products rather than sums). Due to the more stringent barrier, we need something that resembles an Euler product more closely to work with to prove our result. The result is the second half of Lemma 6 in \cite{harper2019partition} with no discretisation applied.
\begin{lemma}
    Uniformly for all $|h| \leq 1$ and $|\sigma-1/2| \leq \frac{1}{\log q}$,
    \begin{align*}
        \mathbb{E}^{\text{char}}&\big|\sum_{\substack{m \leq q^{\varepsilon} \\ P^+(m) \leq P}}\frac{\chi(m)}{m^{1/2+ih}}-\sum_{v=0}^{\lfloor\frac{\varepsilon\log  q}{\log P}\rfloor}\frac{1}{(v!)}\big(\sum_{p^j \leq P}\frac{\chi(p)^j}{jp^{j(1/2+ih)}}\big)^v\big|^2|\sum_{\substack{n \le q^{1/2-2\varepsilon} \\ P^-(n) > P}}\frac{\chi(n)}{n^{\sigma+ih}}|^2 \ll e^{-\frac{\varepsilon\log q}{\log P}}\log q.
    \end{align*} \label{partial_exp}
\end{lemma}
\begin{proof}[Proof of Proposition \ref{conditional_bound_max}]
    To start with, we assume $e^{-U} \leq V \leq e^{\sqrt{\log \log P}}$, otherwise we can ignore the indicator of the log-correlated part of the Dirichlet polynomial being large and proceed in the same way we did in Proposition \ref{good_moment_bound} (or we can proceed as we do here just ignoring all the manipulations to handle the size condition on the large Dirichlet polynomial). Expanding $\log I_{\ell,\chi}(\tilde{h}(\ell))$, we see that $\mathcal{A}_\chi \cap \widetilde{\mathcal{G}}_\chi(h)$ is operating on the quantity
    $$\big|\sum_{p^j \leq P}\frac{\chi(p)^j}{jp^{j(1/2+ih)}}\big|=\big|\sum_{\ell =0}^{\lfloor \log\log P \rfloor -B-1}\log I_{\ell,\chi}(h) +O(1)\big| \le (4C+1)\log\log P + O(1)$$
    where the $O(1)$ term is being contributed by the terms which are $O(p^{-3/2})$. Using another Taylor expansion on the exponential and noting $\frac{\varepsilon\log  q}{\log P} > 21C \log \log P$ (it is $\gg (\log \log  q)^6$, which is more than acceptable for fixed absolute constant $C$) when we apply this to Theorem \ref{short_range_max}, we have
\begin{align*}
    \sum_{v=0}^{\lfloor\frac{\varepsilon\log  q}{\log P}\rfloor}\frac{1}{(v!)}\big(\sum_{\substack{p^j \leq P}}\frac{\chi(p)^j}{jp^{j(1/2+ih)}}\big)^v =& \exp\big(\sum_{\substack{p^j \leq P}}\frac{(\chi(p))^j}{jp^{j(1/2+ih)}}\big) + O\big(\frac{((4C+1)\log \log P)^{\lceil\frac{\varepsilon\log  q}{\log P}\rceil}}{(\left\lceil\frac{\varepsilon\log  q}{\log P}\right\rceil)!}\big) \\
    =&e^{O(1)}\prod_{\ell=0}^{\lfloor \log\log P\rfloor -B-1}I_{\ell,\chi}(h) + O(1).
\end{align*}
Again, the $e^{O(1)}$ term is dependent on fixed constant $B$ and the remainder term is bounded using the Stirling approximation. Why this matters is because we then can firmly connect $\mathcal{A}_\chi \cap \widetilde{\mathcal{G}}_\chi(h)$ to our factored Dirichlet polynomial. Define $\widetilde{H}_\chi(h;V)$ as the event
$$\Big|\prod_{\ell=0}^{\lfloor\log\log P\rfloor -B-1}I_{\ell,\chi}(h)\Big| \ge \frac{\log P}{V}$$
More concretely, there exists a constant $B_0$ dependent on $B$ such that
\begin{align*}
    &\mathbf{1}_{\widetilde{\mathcal{G}}_\chi(h)\cap\mathcal{A}_\chi}(h)\mathbf{1}_{\mathcal{H}_\chi(h;V)}(h)\big|\sum_{\substack{m \leq q^\varepsilon \\ P^+(m) \leq P}}\frac{\chi(m)}{m^{1/2+ih}}\sum_{\substack{n \leq q^{1/2-2\varepsilon} \\ P^-(n) > P}}\frac{\chi(n)}{n^{\sigma+ih}}\big|^2 \numberthis \label{event_translate}\\
    \ll& \mathbf{1}_{\widetilde{\mathcal{G}}_\chi(h) \cap\mathcal{A}_\chi}(h)\mathbf{1}_{\widetilde{\mathcal{H}}_\chi(h;B_0V)}\big|\sum_{\substack{m \leq q^\varepsilon \\ P^+(m)\leq P}}\frac{\chi(m)}{m^{1/2+ih}}\sum_{\substack{n \leq q^{1/2-2\varepsilon} \\ P^-(n) > P}}\frac{\chi(n)}{n^{\sigma+ih}}\big|^2 \\
    +& \big|\sum_{\substack{n \leq q^{1/2-2\varepsilon} \\ P^{-}(n) > P}}\frac{\chi(n)}{n^{\sigma+ih}}\big|^2\big|\sum_{\substack{m \leq q^{\varepsilon} \\ P^+(m) \leq P}}\frac{\chi(m)}{m^{1/2+ih}}-\sum_{v=0}^{\lfloor\frac{\varepsilon\log  q}{\log P}\rfloor}\frac{1}{(v!)}\big(\sum_{p^j \leq P}\frac{\chi(p)^j}{jp^{j(1/2+ih)}}\big)^v\big|^2.
\end{align*}
This follows since if $|\sum_{\substack{m \leq q^\varepsilon \\ P^+(m) \leq P}}\frac{\chi(m)}{m^{1/2+ih}}|>\frac{\log P}{V}$ but $|\prod_{\ell=0}^{S(P)}I_{\ell,\chi}(h)| \leq \frac{\log P}{B_0V}$, where $S(P) = \lfloor\log\log P\rfloor - B -1$, then
\begin{align*}
    &\mathbf{1}_{|\sum_{\substack{m \leq q^\varepsilon \\ P^+(m) \leq P}}\frac{\chi(m)}{m^{1/2+ih}}|>\frac{\log P}{V}}\mathbf{1}_{|\prod_{\ell=0}^{S(P)}I_{\ell,\chi}(h)| \leq \frac{\log P}{B_0V}}\Big|\sum_{\substack{m \le q^\varepsilon \\ P^+(m) \le P}}\frac{\chi(m)}{m^{1/2+ih}}\Big|^2 \\ \ll& \Big|\sum_{\substack{m \le q^\varepsilon \\ P^+(n) \le P}}\frac{\chi(m)}{m^{1/2+ih}} - \sum_{v=0}^{\lfloor\frac{\varepsilon\log  q}{\log P}\rfloor}\frac{1}{(v!)}\big(\sum_{p^j \leq P}\frac{\chi(p)^j}{jp^{j(1/2+ih)}}\big)^v\Big|^2.
\end{align*}
The second term on the right hand side of (\ref{event_translate}) is $\ll e^{-\frac{\varepsilon \log q}{\log P}}\log q$ from Lemma \ref{partial_exp} once averaged over the characters. At this point, we translate the weighted character sum to an expression involving Steinhaus multiplicative functions, with the event $\widetilde{\mathcal{G}}_\chi(h)$ being encoded on the $\tilde{h}(j)$, and $\widetilde{\mathcal{H}}_\chi(h;B_0V)$ being encoded on the value $h$ using Lemma \ref{randomisation}. At this point we have
\begin{equation*}\begin{split}
    &\mathbb{E}^{\text{char}}\bigg[\mathbf{1}_{\mathcal{A}_\chi \cap \widetilde{\mathcal{G}}_\chi(h)}(h)\mathbf{1}_{\tilde{\mathcal{H}}_\chi(h;V)}(h)\big|\sum_{\substack{m \leq q^\varepsilon \\ P^+(m) \leq P}}\frac{\chi(m)}{m^{1/2+ih}}\sum_{\substack{n \leq q^{1/2-2\varepsilon} \\ P^-(n) > P}}\frac{\chi(n)}{n^{\sigma+ih}}\big|^2\bigg] \\
    \ll \ & \mathbb{E}^f\bigg[(1+\delta)^{\log\log P} (\mathbf{1}_{\widetilde{\mathcal{G}}_{f}(h) \cap \widetilde{\mathcal{H}}_f(h;B_0V)} + \delta)\Big|\sum_{\substack{m \leq q^\varepsilon \\ P^+(m) \leq P}}\frac{f(m)}{m^{1/2+ih}}\sum_{\substack{n \leq q^{1/2-2\varepsilon} \\ P^-(n) > P}}\frac{f(n)}{n^{\sigma+ih}}\Big|^2\bigg] \\
    \ll \ & \mathbb{E}^f\bigg[\mathbf{1}_{\widetilde{\mathcal{G}}_{f}(h) \cap \widetilde{\mathcal{H}}_f(h;B_0V)} \Big|\sum_{\substack{m \leq q^\varepsilon \\ P^+(m) \leq P}}\frac{f(m)}{m^{1/2+ih}}\sum_{\substack{n \leq q^{1/2-2\varepsilon} \\ P^-(n) > P}}\frac{f(n)}{n^{\sigma+ih}}\Big|^2\bigg] +\frac{\log q}{(\log\log q)^2}
\end{split}
\end{equation*}
choosing $\delta = (\log\log P)^{-2}$ and $N = 8C\log\log P$ and $\widetilde{\mathcal{G}}_f(h)$ is the event such that for all $0 \le j \le T(P,U)$
$$ \big|\sum_{\ell = j}^{T(P,U)}\log |I_{\ell,f}(\tilde{h}(\ell))|)\big| \leq \log \log P -j + 6\log( \log \log P -j)+U +B'.$$ We define $\widetilde{\mathcal{H}}_f(h;V)$ as the event such that
$$\prod_{\ell=0}^{S(P)}|I_{\ell,f}(h)| \ge \frac{\log P}{Ve^{19/3}}.$$
Note we also discarded $\mathcal{A}_\chi$ in the first inequality since it was sufficient to establish the partition of unity support. What is left is to translate $\widetilde{\mathcal{H}}_f(h;B_0V)$ into information about our shorter Euler product. The first stage is to rewrite this in the form of the random Euler product. Consequently, we have
\begin{equation*}\begin{split}
    &\mathbb{E}^f\bigg[\mathbf{1}_{\widetilde{\mathcal{G}}_{f}(h) \cap \widetilde{\mathcal{H}}_f(h;B_0V)} \Big|\sum_{\substack{m \leq q^\varepsilon \\ P^+(m) \leq P}}\frac{f(m)}{m^{1/2+ih}}\sum_{\substack{n \leq q^{1/2-2\varepsilon} \\ P^-(n) > P}}\frac{f(n)}{n^{\sigma+ih}}\Big|^2\bigg] \\
    \ll  \ & \frac{\log q}{\log P}\Big(\mathbb{E}^f[\mathbf{1}_{\widetilde{\mathcal{G}}_{f}(h) \cap \widetilde{\mathcal{H}}_f(h;B_0V)} |F_P(1/2+ih)|^2]+ \mathbb{E}^f \Big|\sum_{\substack{m \leq q^\varepsilon \\ P^+(m) \leq P}}\frac{f(m)}{m^{1/2+ih}}-F_P(1/2+ih)\Big|^2\Big) \\
    \ll \ &\frac{\log q}{\log P}\mathbb{E}^f[\mathbf{1}_{\widetilde{\mathcal{G}}_{f}(h) \cap \widetilde{\mathcal{H}}_f(h;B_0V)} |F_P(1/2+ih)|^2] + e^{-\frac{\varepsilon\log q}{\log P}}\log q, 
\end{split}
\end{equation*}
by orthogonality of the $f(p)$, so it suffices to understand this tilted expectation using
$$\tilde{\mathbb{P}}_h(A)  = \frac{\mathbb{E}[\mathbf{1}_A|F_P(1/2+ih)|^2]}{\mathbb{E}|F_P(1/2+ih)|^2}$$
and to determine $\tilde{\mathbb{P}}_h(\widetilde{\mathcal{G}}_{f}(h) \cap \widetilde{\mathcal{H}}_f(h;B_0V))$. We need to adapt our event into a shape which is permissible for $\mathbb{E}^f$. Define
$$D(h) :=\sum_{\ell=0}^{S(P)}\log|I_{\ell,f}(h)|-\sum_{\ell=0}^{T(P,U)}\log|I_{\ell,f}(\tilde{h}(\ell))|.$$
We show that $$\tilde{\mathbb{P}}_h(D(h) > 8\log\log\log P) \ll \frac{1}{(\log\log P)^4},$$
which will then allow us to upper bound $\widetilde{\mathcal{H}}_f(h;B_0V)$ with
$$\prod_{\ell=0}^{T(P,U)}|I_{\ell,f}(\tilde{h}(\ell))| \ge \frac{\log P}{B_0 V(\log\log P)^{8}e^{19/3}},$$
which is the correct form for Lemma 4 in \cite{aharper201}. As each $f(p)$ is independent, it suffices to calculate at individual primes. Then
$$\tilde{\mathbb{E}}_h[\prod_{T(P,U) < \ell \le S(P)}|I_{\ell,f}(h)|^2] \ll (2+U)^{12}.$$
For the shifts in $\tilde{h}(\ell)$, then we have using the translation invariance in law of the Steinhaus $f(p)$
\begin{equation*}\begin{split}
    \tilde{\mathbb{E}}_h\bigg[\Big|\frac{(1-f(p)p^{-1/2-i\tilde{h}(\ell)})}{(1-f(p)p^{-1/2-ih})}\Big|^2\bigg] =& \frac{\mathbb{E}[|(1-f(p)p^{-1/2-i(\tilde{h}(\ell)-h)})|^{2}|1-f(p)p^{-1/2}|^{-4}]}{\mathbb{E}[|1-f(p)p^{-1/2}|^{-2}]} \\=& 1+\frac{4(1-\cos((\tilde{h(\ell)}-h)\log(p))}{p(1-p^{-1})^2} \\
    \le& 1+\frac{8(\tilde{h}(\ell)-h)^2\log^2p}{p}.
\end{split}
\end{equation*}
from standard random Euler product estimates. Using the definition of $\tilde{h}(\ell)$, then the spacing between $h$ and $\tilde{h}(\ell)$ is $\ll \frac{e^\ell}{\log P(\log\log P -\ell)^3}$. Summing over the relevant $p$, then we have this is
$$\tilde{\mathbb{E}}_h\bigg[\prod_{\ell=0}^{T(P,U)}\Big|\frac{I_{\ell,f}(h)}{I_{\ell,f}(\tilde{h}(\ell))}\Big|^2\bigg] \ll 1$$ by independence of the $f(p)$ and the summability of $\ell^{-6}$. 
Consequently, we have
$$\tilde{\mathbb{P}}_h(D(h) > 8\log\log\log P) = \tilde{\mathbb{P}}_h(e^{2D(h)} > (\log\log P)^{16}) \ll \frac{1}{(\log\log P)^4}$$
by Chebychev's inequality using $U \le 2\log\log P$. In particular, under the event $D(h) \le 8\log\log\log P$, then
$$\sum_{\ell=0}^{T(P,U)}\log|I_{\ell,f}(\tilde{h}(\ell))| \ge \log\log P -B_1-\log V-8\log\log\log P.$$
We denote the event intersected with $\tilde{\mathcal{G}}_f(h)$ as $\mathcal{E}_f(h)$. It suffices to check this is an appropriate form for Lemma 4 in \cite{aharper201}. The barriers must have absolute value $\le (1/80)\sqrt{e^{-j}\log P}$. This is only a problem at the lower end, and for $j=T(P,U)$, then this restriction has size at least $ (1/80)e^{B/2}(U+2)^2$ which holds for all $U \ge 1$ given $B$ is chosen sufficiently large. Therefore, it remains to bound the expression
\begin{align*}
    \mathbb{P}\big(&\Big\{\sum_{m=1}^jG_m \le a+6g(j) +O(1) \ \forall 1 \le j \le T(P,U)\Big\}\\ \ & \cap\Big\{\sum_{m=1}^{T(P,U)}G_m \ge a+g(T(P,U))-U-\log V - 14\log\log\log P +O(1)  \Big\}\big)
\end{align*}
where $$K=4\log(2+U)+B+O(1), \ g(j)=\log(\frac{K+j}{K+1}), \ a=U+K+6\log K.$$
Then $|g(j)| \le 6\log(1+j)$, so we can apply Proposition \ref{end_point_ballot}, so
$$\frac{\log q}{\log P}\mathbb{E}^f[\mathbf{1}_{\mathcal{E}_f}(h) |F_P(1/2+ih)|^2] \ll \log q\min\{1,\frac{(1+U)}{\sqrt{\log\log q}}\}\Big(\min\{1,\frac{U+\log V +\log\log\log q}{\sqrt{\log\log q}}\}\Big)^2.$$
\end{proof}
\appendix
\section{Gaussian comparison and ballot theorems} \label{prob_heur}
In order to show Proposition \ref{conditional_bound_max} holds, we need a slightly different formulation of the ballot theorem to the one Harper uses in Probability Result 1 in \cite{harper2019partition}. To prove it, we require the following result. Note that this is an asymptotic statement and $\log(1+j) - \log j \le 1/j = O(1)$, so this will translate $a$ by at most a bounded constant.
\begin{proposition}
    (Probability Result 1, \cite{aharper201}) Let $a$ and $n \in \mathbb{N}$ both be large, with $(G_k)_{k=1}^n$ be a sequence of independent real Gaussian random variables with mean $0$ and variance between $1/20$ and $20$, say. Suppose $h$ is a function such that $|h(j)| < 10\log(1+ j)$. Then
    $$\mathbb{P}(\sum_{m=1}^j G_m \leqslant a + h(j),  \ \forall \ 1 \leqslant j \leqslant n ) \asymp \min\big\{1,\frac{a}{\sqrt{n}}\big\}.$$ \label{ballot}
\end{proposition}
\begin{proposition}\label{end_point_ballot}
    Let $a,b$ and $n \in \mathbb{N}$ be sufficiently large, let $(G_k)_{k \leq n}$ be a sequence of independent Gaussian random variables with mean $0$ and variance between $1/20$ and $20$ (say). Then for $|h(j)| \le 10\log(1+j)$ and increasing, we have uniformly
    $$\mathbb{P}(\Big\{\sum_{m=1}^jG_m \leq a +h(j) \ \forall1 \leq j \leq n\Big\} \cap \Big\{a+h(n)-b \leq \sum_{m=1}^nG_m \leq a+h(n)\Big\}) \ll \min\big\{1,\frac{a}{\sqrt{n}}\big\}(\min\big\{1,\frac{b}{\sqrt{n}}\big\})^2.$$
\end{proposition}
\begin{proof}
    This is very short, and follows almost immediately from Proposition \ref{ballot}. We can subtract the first $\lceil 2n/3\rceil$ increments from the full ballot, and we find that
    $$\sum_{m=1}^{j}G_{n-m+1} \ge a+h(n)-b -a-h(n-j) \ge -b, \quad \forall 1 \le j \le \lfloor n/3\rfloor$$
    This then forces the following condition on the middle summands
    $$\sum_{m=\lfloor n/3\rfloor+1}^{n-\lfloor n/3\rfloor}G_m \in [a+h(n)-b-\sum_{m \notin [\lfloor n/3\rfloor+1,n-\lfloor n/3\rfloor]}G_m,a+h(n)-\sum_{m \notin [\lfloor n/3\rfloor+1,n-\lfloor n/3\rfloor]}G_m].$$
    The $G_m$ are independent Gaussians with mean $0$ and variance between $1/20$ and $20$, so by Gaussian concentration estimates, this is $\ll \min\{1,\frac{b}{\sqrt{n}}\}$.
    Noting once again that the $G_m$ are independent and mean $0$, then we have
    \begin{equation*}\begin{split}
        &\mathbb{P}(\Big\{\sum_{m=1}^jG_m \leq a +h(j) \ \forall1 \leq j \leq n\Big\} \cap \Big\{a+h(n)-b \leq \sum_{m=1}^nG_m \leq a+h(n)\Big\}) \\
        \ll \ & \min\{1,\frac{b}{\sqrt{n}}\}\mathbb{P}(\Big\{\sum_{m=1}^jG_m \leq a +h(j) \ \forall1 \leq j \leq \lfloor n/3\rfloor\Big\} \cap \Big\{\sum_{r=1}^{j}(-G_{n-r+1}) \leq b \ \forall 0 \le j \le \lfloor n/3\rfloor\Big\}) \\
        \ll \ & \min\Big\{1,\frac{a}{\sqrt{n}}\Big\}\Big(\min\Big\{1,\frac{b}{\sqrt{n}}\Big\}\Big)^2.
    \end{split}
    \end{equation*}
\end{proof}
\printbibliography[title= References]
\end{document}